\documentclass{article}

\usepackage{amsmath,amssymb}

\usepackage{diagbox}

\usepackage{mathrsfs}

\usepackage[amsmath, thmmarks]{ntheorem}

\usepackage{hyperref}

\hypersetup{colorlinks=true,allcolors=[rgb]{0,0.4,1},citecolor=red}

\usepackage{tikz}
\usepackage{graphicx}
\usetikzlibrary{calc}

\usepackage{array}
\usepackage{makecell}

\newtheorem{theorem}{Theorem}[section]
\newtheorem{lemma}[theorem]{Lemma}
\newtheorem{corollary}[theorem]{Corollary}

\newtheorem{claim}[theorem]{Claim}
\newtheorem{conjecture}[theorem]{Conjecture}

\theorembodyfont{\normalfont}
\newtheorem{definition}[theorem]{Definition}
\newtheorem*{remark}{Remark}
\newtheorem{example}[theorem]{Example}

\theoremheaderfont{\itshape}
\theoremseparator{.}
\theoremsymbol{\ensuremath{\blacksquare}}
\newtheorem*{proof}{Proof}

\newcommand{\set}[1]{\{\,#1\,\}}

\newcommand{\BA}{\mathbb{A}}
\newcommand{\BB}{\mathbb{B}}
\newcommand{\BC}{\mathbb{C}}
\newcommand{\BD}{\mathbb{D}}
\newcommand{\BE}{\mathbb{E}}
\newcommand{\BF}{\mathbb{F}}
\newcommand{\BG}{\mathbb{G}}
\newcommand{\BH}{\mathbb{H}}
\newcommand{\BI}{\mathbb{I}}
\newcommand{\BJ}{\mathbb{J}}
\newcommand{\BK}{\mathbb{K}}
\newcommand{\BL}{\mathbb{L}}
\newcommand{\BM}{\mathbb{M}}
\newcommand{\BN}{\mathbb{N}}
\newcommand{\BO}{\mathbb{O}}
\newcommand{\BP}{\mathbb{P}}
\newcommand{\BQ}{\mathbb{Q}}
\newcommand{\BR}{\mathbb{R}}
\newcommand{\BS}{\mathbb{S}}
\newcommand{\BT}{\mathbb{T}}
\newcommand{\BU}{\mathbb{U}}
\newcommand{\BV}{\mathbb{V}}
\newcommand{\BW}{\mathbb{W}}
\newcommand{\BX}{\mathbb{X}}
\newcommand{\BY}{\mathbb{Y}}
\newcommand{\BZ}{\mathbb{Z}}

\newcommand{\bA}{\mathbf{A}}
\newcommand{\cA}{\mathcal{A}}
\newcommand{\fA}{\mathfrak{A}}
\newcommand{\rA}{\mathrm{A}}
\newcommand{\sA}{\mathscr{A}}
\newcommand{\bB}{\mathbf{B}}
\newcommand{\cB}{\mathcal{B}}
\newcommand{\fB}{\mathfrak{B}}
\newcommand{\rB}{\mathrm{B}}
\newcommand{\sB}{\mathscr{B}}
\newcommand{\bC}{\mathbf{C}}
\newcommand{\cC}{\mathcal{C}}
\newcommand{\fC}{\mathfrak{C}}
\newcommand{\rC}{\mathrm{C}}
\newcommand{\sC}{\mathscr{C}}
\newcommand{\bD}{\mathbf{D}}
\newcommand{\cD}{\mathcal{D}}
\newcommand{\fD}{\mathfrak{D}}
\newcommand{\rD}{\mathrm{D}}
\newcommand{\sD}{\mathscr{D}}
\newcommand{\bE}{\mathbf{E}}
\newcommand{\cE}{\mathcal{E}}
\newcommand{\fE}{\mathfrak{E}}
\newcommand{\rE}{\mathrm{E}}
\newcommand{\sE}{\mathscr{E}}
\newcommand{\bF}{\mathbf{F}}
\newcommand{\cF}{\mathcal{F}}
\newcommand{\fF}{\mathfrak{F}}
\newcommand{\rF}{\mathrm{F}}
\newcommand{\sF}{\mathscr{F}}
\newcommand{\bG}{\mathbf{G}}
\newcommand{\cG}{\mathcal{G}}
\newcommand{\fG}{\mathfrak{G}}
\newcommand{\rG}{\mathrm{G}}
\newcommand{\sG}{\mathscr{G}}
\newcommand{\bH}{\mathbf{H}}
\newcommand{\cH}{\mathcal{H}}
\newcommand{\fH}{\mathfrak{H}}
\newcommand{\rH}{\mathrm{H}}
\newcommand{\sH}{\mathscr{H}}
\newcommand{\bI}{\mathbf{I}}
\newcommand{\cI}{\mathcal{I}}
\newcommand{\fI}{\mathfrak{I}}
\newcommand{\rI}{\mathrm{I}}
\newcommand{\sI}{\mathscr{I}}
\newcommand{\bJ}{\mathbf{J}}
\newcommand{\cJ}{\mathcal{J}}
\newcommand{\fJ}{\mathfrak{J}}
\newcommand{\rJ}{\mathrm{J}}
\newcommand{\sJ}{\mathscr{J}}
\newcommand{\bK}{\mathbf{K}}
\newcommand{\cK}{\mathcal{K}}
\newcommand{\fK}{\mathfrak{K}}
\newcommand{\rK}{\mathrm{K}}
\newcommand{\sK}{\mathscr{K}}
\newcommand{\bL}{\mathbf{L}}
\newcommand{\cL}{\mathcal{L}}
\newcommand{\fL}{\mathfrak{L}}
\newcommand{\rL}{\mathrm{L}}
\newcommand{\sL}{\mathscr{L}}
\newcommand{\bM}{\mathbf{M}}
\newcommand{\cM}{\mathcal{M}}
\newcommand{\fM}{\mathfrak{M}}
\newcommand{\rM}{\mathrm{M}}
\newcommand{\sM}{\mathscr{M}}
\newcommand{\bN}{\mathbf{N}}
\newcommand{\cN}{\mathcal{N}}
\newcommand{\fN}{\mathfrak{N}}
\newcommand{\rN}{\mathrm{N}}
\newcommand{\sN}{\mathscr{N}}
\newcommand{\bO}{\mathbf{O}}
\newcommand{\cO}{\mathcal{O}}
\newcommand{\fO}{\mathfrak{O}}
\newcommand{\rO}{\mathrm{O}}
\newcommand{\sO}{\mathscr{O}}
\newcommand{\bP}{\mathbf{P}}
\newcommand{\cP}{\mathcal{P}}
\newcommand{\fP}{\mathfrak{P}}
\newcommand{\rP}{\mathrm{P}}
\newcommand{\sP}{\mathscr{P}}
\newcommand{\bQ}{\mathbf{Q}}
\newcommand{\cQ}{\mathcal{Q}}
\newcommand{\fQ}{\mathfrak{Q}}
\newcommand{\rQ}{\mathrm{Q}}
\newcommand{\sQ}{\mathscr{Q}}
\newcommand{\bR}{\mathbf{R}}
\newcommand{\cR}{\mathcal{R}}
\newcommand{\fR}{\mathfrak{R}}
\newcommand{\rR}{\mathrm{R}}
\newcommand{\sR}{\mathscr{R}}
\newcommand{\bS}{\mathbf{S}}
\newcommand{\cS}{\mathcal{S}}
\newcommand{\fS}{\mathfrak{S}}
\newcommand{\rS}{\mathrm{S}}
\newcommand{\sS}{\mathscr{S}}
\newcommand{\bT}{\mathbf{T}}
\newcommand{\cT}{\mathcal{T}}
\newcommand{\fT}{\mathfrak{T}}
\newcommand{\rT}{\mathrm{T}}
\newcommand{\sT}{\mathscr{T}}
\newcommand{\bU}{\mathbf{U}}
\newcommand{\cU}{\mathcal{U}}
\newcommand{\fU}{\mathfrak{U}}
\newcommand{\rU}{\mathrm{U}}
\newcommand{\sU}{\mathscr{U}}
\newcommand{\bV}{\mathbf{V}}
\newcommand{\cV}{\mathcal{V}}
\newcommand{\fV}{\mathfrak{V}}
\newcommand{\rV}{\mathrm{V}}
\newcommand{\sV}{\mathscr{V}}
\newcommand{\bW}{\mathbf{W}}
\newcommand{\cW}{\mathcal{W}}
\newcommand{\fW}{\mathfrak{W}}
\newcommand{\rW}{\mathrm{W}}
\newcommand{\sW}{\mathscr{W}}
\newcommand{\bX}{\mathbf{X}}
\newcommand{\cX}{\mathcal{X}}
\newcommand{\fX}{\mathfrak{X}}
\newcommand{\rX}{\mathrm{X}}
\newcommand{\sX}{\mathscr{X}}
\newcommand{\bY}{\mathbf{Y}}
\newcommand{\cY}{\mathcal{Y}}
\newcommand{\fY}{\mathfrak{Y}}
\newcommand{\rY}{\mathrm{Y}}
\newcommand{\sY}{\mathscr{Y}}
\newcommand{\bZ}{\mathbf{Z}}
\newcommand{\cZ}{\mathcal{Z}}
\newcommand{\fZ}{\mathfrak{Z}}
\newcommand{\rZ}{\mathrm{Z}}
\newcommand{\sZ}{\mathscr{Z}}

\newcommand{\ba}{\mathbf{a}}
\newcommand{\fa}{\mathfrak{a}}
\newcommand{\ra}{\mathrm{a}}
\newcommand{\bb}{\mathbf{b}}
\newcommand{\fb}{\mathfrak{b}}
\newcommand{\rb}{\mathrm{b}}
\newcommand{\bc}{\mathbf{c}}
\newcommand{\fc}{\mathfrak{c}}
\newcommand{\rc}{\mathrm{c}}
\newcommand{\bd}{\mathbf{d}}
\newcommand{\fd}{\mathfrak{d}}
\newcommand{\rd}{\mathrm{d}}
\newcommand{\be}{\mathbf{e}}
\newcommand{\fe}{\mathfrak{e}}
\newcommand{\re}{\mathrm{e}}
\newcommand{\ff}{\mathfrak{f}}
\newcommand{\rf}{\mathrm{f}}
\newcommand{\bg}{\mathbf{g}}
\newcommand{\fg}{\mathfrak{g}}
\newcommand{\rg}{\mathrm{g}}
\newcommand{\bh}{\mathbf{h}}
\newcommand{\fh}{\mathfrak{h}}
\newcommand{\rh}{\mathrm{h}}
\newcommand{\bi}{\mathbf{i}}
\newcommand{\ri}{\mathrm{i}}
\newcommand{\bj}{\mathbf{j}}
\newcommand{\fj}{\mathfrak{j}}
\newcommand{\rj}{\mathrm{j}}
\newcommand{\bk}{\mathbf{k}}
\newcommand{\fk}{\mathfrak{k}}
\newcommand{\rk}{\mathrm{k}}
\newcommand{\bl}{\mathbf{l}}
\newcommand{\fl}{\mathfrak{l}}
\newcommand{\rl}{\mathrm{l}}
\newcommand{\bm}{\mathbf{m}}
\newcommand{\fm}{\mathfrak{m}}
\newcommand{\bn}{\mathbf{n}}
\newcommand{\fn}{\mathfrak{n}}
\newcommand{\rn}{\mathrm{n}}
\newcommand{\bo}{\mathbf{o}}
\newcommand{\fo}{\mathfrak{o}}
\newcommand{\ro}{\mathrm{o}}
\newcommand{\bp}{\mathbf{p}}
\newcommand{\fp}{\mathfrak{p}}
\newcommand{\rp}{\mathrm{p}}
\newcommand{\bq}{\mathbf{q}}
\newcommand{\fq}{\mathfrak{q}}
\newcommand{\br}{\mathbf{r}}
\newcommand{\fr}{\mathfrak{r}}
\newcommand{\rr}{\mathrm{r}}
\newcommand{\bs}{\mathbf{s}}
\newcommand{\fs}{\mathfrak{s}}
\newcommand{\rs}{\mathrm{s}}
\newcommand{\bt}{\mathbf{t}}
\newcommand{\ft}{\mathfrak{t}}
\newcommand{\rt}{\mathrm{t}}
\newcommand{\bu}{\mathbf{u}}
\newcommand{\fu}{\mathfrak{u}}
\newcommand{\ru}{\mathrm{u}}
\newcommand{\bv}{\mathbf{v}}
\newcommand{\fv}{\mathfrak{v}}
\newcommand{\rv}{\mathrm{v}}
\newcommand{\bw}{\mathbf{w}}
\newcommand{\fw}{\mathfrak{w}}
\newcommand{\rw}{\mathrm{w}}
\newcommand{\bx}{\mathbf{x}}
\newcommand{\fx}{\mathfrak{x}}
\newcommand{\rx}{\mathrm{x}}
\newcommand{\by}{\mathbf{y}}
\newcommand{\fy}{\mathfrak{y}}
\newcommand{\ry}{\mathrm{y}}
\newcommand{\bz}{\mathbf{z}}
\newcommand{\fz}{\mathfrak{z}}
\newcommand{\rz}{\mathrm{z}}

\DeclareMathOperator{\End}{End}
\newcommand{\SL}{\mathfrak{sl}}

\title{On values of $\SL_3$ weight system on chord diagrams whose intersection graph is complete bipartite}
\author{Zhuoke Yang \\ \textsf{zetadkyzk@gmail.com}}
\date{}

\begin{document}
\maketitle 
\begin{abstract}
Each knot invariant can be extended to singular knots according to the skein rule.	
A Vassiliev invariant of order at most $n$ is defined as a knot invariant that vanishes identically on knots with more than $n$ double points. A chord diagram encodes the order of double points along a singular knot. A Vassiliev invariant of order $n$ gives rise to a function on chord diagrams with $n$ chords.  
Such a function should satisfy some conditions in order to come from a Vassiliev invariant. A weight system is a function on chord diagrams that satisfies so-called 4-term relations. Given a Lie algebra $\fg$ equipped with a non-degenerate invariant bilinear form, one can construct a weight system with values in the center of the universal enveloping algebra $U(\fg)$. In this paper, we calculate $\SL_3$ weight system for chord diagram whose intersection graph is complete bipartite graph $K_{2,n}$.
\end{abstract}

\section{Introduction}\label{s1}
Finite order knot invariants, which were introduced in \cite{vassiliev1990cohomology} by V. Vassiliev near 1990, may be expressed in terms of weight systems, that is, functions on chord diagrams satisfying the so-called Vassiliev $4$-term relations. In paper \cite{kon1993}, M. Kontsevich proved that over a field of characteristic zero every weight system corresponds to some finite order invariant. There are multiple approaches to constructing weight systems. In particular, D. Bar-Natan and M. Kontsevtch suggested a construction of a weight system coming from a finite dimensional Lie algebra endowed with an invariant nondegenerate bilinear form. The $\SL_2$ Lie algebra weight system is the simplest case. Its values lie in the center of the universal enveloping algebra of the $\SL_2$ Lie algebra, which, in turn, is isomorphic to the ring of polynomials in one variable (the Casimir element). The $\SL_2$ weight system was studied in many papers. Despite the fact that this weight system can be defined easily, it is difficult to compute its value on a chord diagram using the definition because it is necessary to work with elements of a non-commutative algebra in order to do this. The Chmutov–Varchenko recurrence relations \cite{chmutov1997remarks} simplify these computations significantly. A theorem by S.~Chmutov and S.~Lando \cite{Chmutov2007} states that the value of the $\SL_2$ weight system on a chord diagram depends only on the intersection graph of this chord diagram, i.e. if two chord diagrams have isomorphic intersection graphs, the values of the weight system on these chord diagrams coincide. 

On the other side, we don't have such good properties for
the next reasonable case, namely, for the $\SL_3$ weight system. The $\SL_3$ Lie algebra weight system takes values in the center of the universal enveloping algebra of the $\SL_3$ Lie algebra which is isomorphic to the ring of polynomials in TWO variables (the Casimir elements of degrees 2 and 3). For the $\SL_3$ weight system, we do
not have a result similar to the Chmutov–-Varchenko recurrence relations for $\SL_2$ weight system which could help us to compute its value. The Chmutov--Lando theorem also fails on $\SL_3$ weight system, which means there are two different chord diagrams with different values in $\SL_3$ weight system such that they have isomorphic intersection graphs. 

Our main results concern explicit values of the $\SL_3$ weight system
on chord diagrams whose intersection graph is complete bipartite,
with the size of one part equal to~$2$. In our computations, we
use the results from~\cite{yoshizumi1999some}.
Up to now, the explicit values of the $\SL_3$ weight system has been
known only in few examples and simple series. Our results imply
a nontrivial conclusion that for the chord diagrams whose intersection
graph is the complete bipartite graph $K_{2,n}$, the value
of the $\SL_3$ weight system depends on the second Casimir only.

A key role in our study is played by the Hopf algebra structure
on the space of chord diagrams modulo $4$-term relations introduced
by Kontsevich. Chord diagrams whose intersection graph is complete
bipartite generate a Hopf subalgebra in this Hopf algebra.
By analyzing the structure of this Hopf subalgebra,
P.~Filippova managed in~\cite{Filippova2020,Filippova2021} to deduce
the values of the $\SL_2$ weight system on projections of the
chord diagrams whose intersection graph is complete bipartite
to the subspace of primitives. By combining our computations
with her results, we obtain explicit expressions for the values
on primitives of the $\SL_3$ weight system.

The paper is organized as follows. In Sec.~\ref{s2}, we give definitions of the Hopf algebras of chord diagrams and Lie algebra weight system. In Sec.~\ref{s3}, we show the way to calculate $\SL_3$ weight system and
give some results on small chord diagrams. In Sec.~\ref{s4}, we give the definitions of the Hopf algebras of Jacobi diagrams and calculate some results on Jacobi diagrams. In Sec.~\ref{s5}, we prove the major theorem which is the formula for
the values of the $\SL_3$ weight system on chord diagrams whose intersection graph is complete bipartite and their projections to primitives.

In our presentation, we follow the approach of \cite{lando2013graphs}
to chord diagrams and weight systems, see also \cite{Chmutov2012}.

\section{Hopf algebras of chord diagrams and Lie algebra weight systems}\label{s2}
In this section we define the Hopf algebra of chord diagrams modulo $4$-term relations. 
\begin{definition}[chord diagram]
A {\it chord diagram} $D$ of order $n$ (or degree $n$) is an oriented circle (sometimes termed {\it Wilson loop})
 with a distinguished set of $n$ disjoint pairs of distinct points, considered up to orientation preserving diffeomorphisms of the circle. We denote the set of chords of a chord diagram $D$ by $[D]$.
\end{definition}

The vector space $\cA$ spanned by chord diagrams over complex field $\BC$ is graded, 
$$
\cA=\cA_0\oplus\cA_1\oplus\cA_2\oplus\cA_3\oplus\dots.
$$
 Each component~$\cA_n$ is spanned by diagrams of the same order~$n$.

\begin{definition}[$4$-term elements]
A $4$-term (or $4T$) element is the alternating sum of the following quadruples of diagrams:

\  \  \  \  \  \ 
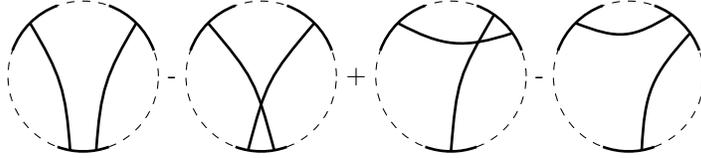
\begin{tikzpicture}[baseline={([yshift=-.5ex]current bounding box.center)}]
	\draw[dashed] (0,0) circle (1);
	\draw[line width=1pt]  ([shift=( 20:1cm)]0,0) arc [start angle= 20, end angle= 70, radius=1];
	\draw[line width=1pt]  ([shift=(110:1cm)]0,0) arc [start angle=110, end angle=160, radius=1];
	\draw[line width=1pt]  ([shift=(250:1cm)]0,0) arc [start angle=250, end angle=290, radius=1];
	\draw[line width=1pt] (45:1) ..  controls (5:0.3) and (-40:0.3)  .. (280:1);
	\draw[line width=1pt] (135:1) ..  controls (175:0.3) and (220:0.3)  .. (260:1);
\end{tikzpicture}  - 
\begin{tikzpicture}[baseline={([yshift=-.5ex]current bounding box.center)}]
	\draw[dashed] (0,0) circle (1);
	\draw[line width=1pt]  ([shift=( 20:1cm)]0,0) arc [start angle= 20, end angle= 70, radius=1];
	\draw[line width=1pt]  ([shift=(110:1cm)]0,0) arc [start angle=110, end angle=160, radius=1];
	\draw[line width=1pt]  ([shift=(250:1cm)]0,0) arc [start angle=250, end angle=290, radius=1];
	\draw[line width=1pt] (45:1) ..  controls (-5:0.1) and (-50:0.1)  .. (260:1);
	\draw[line width=1pt] (135:1) ..  controls (185:0.1) and (225:0.1)  .. (280:1);
\end{tikzpicture}  + 
\begin{tikzpicture}[baseline={([yshift=-.5ex]current bounding box.center)}]
	\draw[dashed] (0,0) circle (1);
	\draw[line width=1pt]  ([shift=( 20:1cm)]0,0) arc [start angle= 20, end angle= 70, radius=1];
	\draw[line width=1pt]  ([shift=(110:1cm)]0,0) arc [start angle=110, end angle=160, radius=1];
	\draw[line width=1pt]  ([shift=(250:1cm)]0,0) arc [start angle=250, end angle=290, radius=1];
	\draw[line width=1pt] (55:1) ..  controls (5:0.1) and (-40:0.1)  .. (270:1);
	\draw[line width=1pt] (135:1) ..  controls (105:0.4) and (65:0.4)  .. (35:1);
\end{tikzpicture}  - 
\begin{tikzpicture}[baseline={([yshift=-.5ex]current bounding box.center)}]
	\draw[dashed] (0,0) circle (1);
	\draw[line width=1pt]  ([shift=( 20:1cm)]0,0) arc [start angle= 20, end angle= 70, radius=1];
	\draw[line width=1pt]  ([shift=(110:1cm)]0,0) arc [start angle=110, end angle=160, radius=1];
	\draw[line width=1pt]  ([shift=(250:1cm)]0,0) arc [start angle=250, end angle=290, radius=1];
	\draw[line width=1pt] (35:1) ..  controls (0:0.3) and (-45:0.3)  .. (280:1);
	\draw[line width=1pt] (135:1) ..  controls (105:0.5) and (85:0.5)  .. (55:1);
\end{tikzpicture}.
\end{definition}
Here all the four chord diagrams contain, in addition to the two depicted chords,
one and the same set of other chords.
For any vector space $V$, a function $f\in \hom_{linear}(\cA,V)$ 
that vanishes on all $4$-term elements is called a {\it  weight system}.

Now we define the Hopf algebra structure on $\cA/\langle4T\rangle:=\cA^{fr}$.

\begin{definition}
The {\it product} of two chord diagrams $D_1$ and $D_2$ is defined by cutting and gluing the two circles as shown

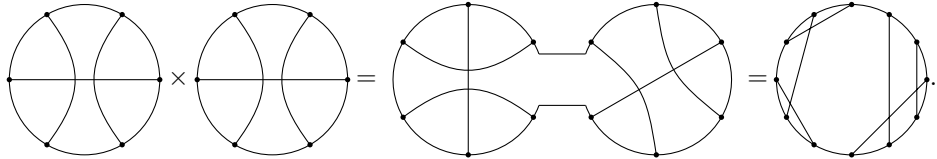
\begin{tikzpicture}[baseline={([yshift=-.5ex]current bounding box.center)}]
	\draw (0,0) circle (1);
	\draw (1,0) -- (-1,0);
	\fill[black] (1,0) circle (1pt)
                 (-1,0) circle (1pt)
                 (60:1) circle (1pt)
                 (120:1) circle (1pt)
                 (240:1) circle (1pt)
                 (300:1) circle (1pt);
	\draw (60:1)  ..  controls (0,0.2) and (0,-0.2)  .. (300:1);
	\draw (120:1) ..  controls (0,0.2) and (0,-0.2)  .. (240:1);
\end{tikzpicture} $\times$ 
\begin{tikzpicture}[baseline={([yshift=-.5ex]current bounding box.center)}]
	\draw (0,0) circle (1);
	\draw (1,0) -- (-1,0);
	\fill[black] (1,0) circle (1pt)
                 (-1,0) circle (1pt)
                 (60:1) circle (1pt)
                 (120:1) circle (1pt)
                 (240:1) circle (1pt)
                 (300:1) circle (1pt);
	\draw (60:1)  ..  controls (0,0.2) and (0,-0.2)  .. (300:1);
	\draw (120:1) ..  controls (0,0.2) and (0,-0.2)  .. (240:1);
\end{tikzpicture} $=$ 
\begin{tikzpicture}[baseline={([yshift=-.5ex]current bounding box.center)}]
	\draw ([shift=( 20:1cm)]0,0) arc [start angle= 20, end angle= 340, radius=1];;
	\draw (0,1) -- (0,-1);
	\fill[black] (0,1) circle (1pt)
                 (0,-1) circle (1pt)
                 (30:1) circle (1pt)
                 (150:1) circle (1pt)
                 (210:1) circle (1pt)
                 (330:1) circle (1pt);
	\draw (30:1)  ..  controls (0.2,0) and (-0.2,0)  .. (150:1);
	\draw (210:1) ..  controls (-0.2,0) and (0.2,0)  .. (330:1);

	\draw[xshift=2.5cm] ([shift=( 200:1cm)]0,0) arc [start angle= 200, end angle=520, radius=1];
	\draw[xshift=2.5cm] (30:1) -- (210:1);
	\fill[xshift=2.5cm,black] (0,1) circle (1pt)
                 (0,-1) circle (1pt)
                 (30:1) circle (1pt)
                 (150:1) circle (1pt)
                 (210:1) circle (1pt)
                 (330:1) circle (1pt);
	\draw[xshift=2.5cm] (150:1)  ..  controls (190:0.2) and (230:0.2)  .. (0,-1);
	\draw[xshift=2.5cm] (0,1) ..  controls (50:0.2) and (10:0.2)  .. (330:1);
	\draw (20:1) -- ($(20:1)+(0.62,0)$);
	\draw (-20:1) -- ($(-20:1)+(0.62,0)$);
\end{tikzpicture} $=$ 
\begin{tikzpicture}[baseline={([yshift=-.5ex]current bounding box.center)}]
	\draw (0,0) circle (1);
	\draw (  0:1) -- (-90:1)
          ( 30:1) -- (-30:1)
          ( 60:1) -- (-60:1)
          ( 90:1) -- (150:1)
          (120:1) -- (210:1)
          (180:1) -- (240:1);
	\fill[black] (  0:1) circle (1pt)
                 ( 30:1) circle (1pt)
                 ( 60:1) circle (1pt)
                 ( 90:1) circle (1pt)
                 (120:1) circle (1pt)
                 (150:1) circle (1pt)
                 (180:1) circle (1pt)
                 (210:1) circle (1pt)
                 (240:1) circle (1pt)
                 (270:1) circle (1pt)
                 (300:1) circle (1pt)
                 (330:1) circle (1pt);
\end{tikzpicture}.

Modulo $4$-term relationship, the product is well-defined.
\end{definition}
\begin{definition}
The {\it coproduct} in the algebra $\cA^{fr}$
\[
	\delta:\cA_n^{fr}\to \bigoplus_{k+l=n}\cA_k^{fr}\otimes\cA_l^{fr}
\]
is defined as follows. For a diagram $D \in \cA_n^{fr}$ we put
\[
	\delta(D):=\sum_{J\subseteq[D]} D_J\otimes D_{\bar{J}}.
\]
The summation is taken over all subsets $J$ of the set of chords of $D$. Here $D_J$ is the diagram consisting of the chords that belong to $J$ and $\bar{J} = [D] \setminus J$ is the complementary subset of chords. To the entire space $\cA^{fr}$, the operator $\delta$ is extended by linearity.
\end{definition}
\begin{claim} The vector space $A^{fr}$ 
	endowed with the above product and coproduct 
	is a commutative, cocommutative and connected bialgebra.
\end{claim}
\begin{definition}An element $p$ of a bialgebra is called {\it primitive} if $\delta(p) = 1 \otimes p + p \otimes 1$.
\end{definition}
It is easy to show that primitive elements form a 
vector subspace $P(\cA^{fr})$ in the bialgebra $\cA^{fr}$. Since any homogeneous component of a primitive element is primitive, such a vector subspace of a graded bialgebra is also graded $P_n=P(\cA^{fr})\cap\cA_n^{fr}$.
An element of~$\cA_n$ is {\it decomposable} if it can be represented as 
a product of elements of order smaller than~$n$.

\begin{theorem}[\cite{lando2000hopf,schmitt1994incidence}]
The projection $\pi(D)$ of a graph $D$ to the subspace of primitive elements whose kernel is the subspace spanned by decomposable elements in the Hopf algebra $D$ is given by the formula
\begin{align*}
	\pi(D) &= D-1!\sum_{[D_1]\sqcup [D_2]=[D]}D_1\cdot D_2 +2! \sum_{[D_1]\sqcup [D_2]\sqcup[D_3]=[D]}D_1\cdot D_2\cdot D_3\dots\\
	&= D-\sum_{i=2}^{|[D]|}(-1)^i(i-1)!\sum_{\substack{\bigsqcup\limits_{j=1}^i [D_j]=[D]\\ [D_j]\ne \emptyset}}\prod_{j=1}^i D_j
\end{align*}
\end{theorem}

For example
\begin{example} The element \\  
\[\pi(\begin{tikzpicture}[baseline={([yshift=-.5ex]current bounding box.center)},scale=0.3]
	\draw (0,0) circle (1);
	\draw (1,0) -- (-1,0);
	\fill[black] (1,0) circle (3pt)
                 (-1,0) circle (3pt)
                 (60:1) circle (3pt)
                 (120:1) circle (3pt)
                 (240:1) circle (3pt)
                 (300:1) circle (3pt);
	\draw (60:1)  ..  controls (0,0.2) and (0,-0.2)  .. (300:1);
	\draw (120:1) ..  controls (0,0.2) and (0,-0.2)  .. (240:1);
\end{tikzpicture})=
\begin{tikzpicture}[baseline={([yshift=-.5ex]current bounding box.center)},scale=0.3]
	\draw (0,0) circle (1);
	\draw (1,0) -- (-1,0);
	\fill[black] (1,0) circle (3pt)
                 (-1,0) circle (3pt)
                 (60:1) circle (3pt)
                 (120:1) circle (3pt)
                 (240:1) circle (3pt)
                 (300:1) circle (3pt);
	\draw (60:1)  ..  controls (0,0.2) and (0,-0.2)  .. (300:1);
	\draw (120:1) ..  controls (0,0.2) and (0,-0.2)  .. (240:1);
\end{tikzpicture}-2
\begin{tikzpicture}[baseline={([yshift=-.5ex]current bounding box.center)},scale=0.3]
	\draw (0,0) circle (1);
	\fill[black] (0:1) circle (3pt)
                 (180:1) circle (3pt)
                 (60:1) circle (3pt)
                 (120:1) circle (3pt)
                 (210:1) circle (3pt)
                 (330:1) circle (3pt);
	\draw (60:1)--(180:1);
	\draw (0:1)--(120:1);
	\draw (210:1)--(330:1);
\end{tikzpicture}+
\begin{tikzpicture}[baseline={([yshift=-.5ex]current bounding box.center)},scale=0.3]
	\draw (0,0) circle (1);
	\draw (1,0) -- (-1,0);
	\fill[black] (1,0) circle (3pt)
                 (-1,0) circle (3pt)
                 (45:1) circle (3pt)
                 (135:1) circle (3pt)
                 (225:1) circle (3pt)
                 (315:1) circle (3pt);
	\draw (45:1)--(135:1);
	\draw (225:1)--(315:1);
\end{tikzpicture}
\]
is a primitive element, which is the projection of the argument
in the left-hand side to the subspace of primitives.
\end{example}
Given a Lie algebra $\fg$ equipped with a non-degenerate invariant bilinear form, one can construct a weight system with values in the center of its universal enveloping algebra $U(\fg)$.  These construction is due to M. Kontsevich~\cite{kon1993} and D. Bar-Natan~\cite{bar1995vassiliev}.
\begin{definition}[Universal Lie algebra weight systems]Kontsevich’s construction proceeds as follows. Let $\fg$ be a metrized Lie algebra over $\mathbb{R}$ or $\mathbb{C}$, that is, a Lie algebra with an ad-invariant non-degenerate bilinear form $\langle\cdot,\cdot\rangle$. 
	Let~$m$ denote the dimension of~$\fg$. Choose a basis $e_1 ,\dots,e_m$ of $\fg$ and let $e_1^* ,\dots,e_m^*$ be the dual basis with respect to the form $\langle\cdot,\cdot\rangle$. 

Given a chord diagram $D$ with $n$ chords, we first choose a base point on the circle, away from the ends of the chords of $D$. This gives a linear order on the endpoints of the chords, increasing in the positive direction of the Wilson loop. Assign to each chord $a$ an index, that is, an integer-valued variable, $i_a$. The values of $i_a$ will range from $1$ to $m$, the dimension of the Lie algebra. Mark the first endpoint of the chord with the symbol $e_{i_a}$ and the second endpoint with $e_{i_a}^{*}$.

Now, write the product of all the $e_{i_a}$ and all the $e_{i_a}^{*}$ , in the order in which they appear on the Wilson loop of $D$, and take the sum of the $m^n$ elements of the universal enveloping algebra $U(\fg)$ obtained by substituting all possible values of the indices $i_a$ into this product. Denote by $\phi_\fg (D)$ the resulting element of $U(\fg)$.
\end{definition}

\begin{claim}The above construction has the following properties:
\begin{enumerate}
\item the element $\phi_\fg (D)$ does not depend on the choice of the base point on the diagram;
\item it does not depend on the choice of the basis ${e_i }$ of the Lie algebra;
\item it belongs to the ad-invariant subspace
\[
	U(\fg)^{\fg}=\{x\in U(\fg)|xy=yx \text{ for all } y\in \fg\}=ZU(\fg).
\]
\item This map from chord diagrams to $ZU(\fg)$ satisfies the 4-term relations. Therefore, it extends to a homomorphism of commutative algebras
$\cA^{fr}\to ZU(\fg)$.
\end{enumerate}
\end{claim}

\begin{remark}
If $D$ is a chord diagram with $n$ chords, then
\[
	\phi_\fg (D)= c^n + \{\text{terms of degree less than } 2n \text{ in }U(g)\},
\]
where $c=e_1\otimes e_1^*+\dots+e_m\otimes e_m^*\in U(\fg)$ is the quadratic Casimir element. Indeed, we can permute the endpoints of chords on the circle without changing the highest term of $\phi_\fg (D)$ since all the additional summands arising as commutators have degrees smaller than $2n$. Therefore, the highest degree term of $\phi_\fg (D)$ does not depend on $D$ with a given number~$n$ of chords. Finally, if $D$ is a diagram with $n$ isolated chords, that is, the $n$th power of the diagram with one chord, then $\phi_\fg (D) = c^n $.
\end{remark}

\section{The $\SL_3$ weight system}\label{s3}
In this section, we concentrate on the weight
system associated to the Lie algebra $\SL_3$.

\begin{definition}[Weight systems associated with representations]
A linear representation $T : \fg \to \End(V)$ extends to a homomorphism of associative algebras $U(T) : U(\fg) \to \End(V )$. The composition of following three maps (with the last map being the trace)
\[
	\cA\xrightarrow{\phi_\fg} U(\fg)\xrightarrow{U(T)} \End(V)\xrightarrow{Tr} \mathbb{C}
\]
by definition gives the {\it weight system associated with the representation}
\[
	\phi_\fg^{T}=Tr\circ U(T) \circ \phi_\fg
\]
\end{definition}

Consider the standard representation of the Lie algebra $\SL_3$ 
as the space of $3\times3$ matrices with zero trace. It is an eight-dimensional Lie algebra spanned by the matrices
\begin{align*}
&E_1 = \begin{pmatrix}0& 1& 0\\ 0& 0& 0\\ 0& 0& 0\end{pmatrix},&
&E_2 = \begin{pmatrix}0& 0& 0\\ 0& 0& 1\\ 0& 0& 0\end{pmatrix},&
&E_3 = \begin{pmatrix}0& 0& 1\\ 0& 0& 0\\ 0& 0& 0\end{pmatrix},&
&H_1 = \begin{pmatrix}1& 0& 0\\ 0& -1& 0\\ 0& 0& 0\end{pmatrix},&\\
&F_1 = \begin{pmatrix}0& 0& 0\\ 1& 0& 0\\ 0& 0& 0\end{pmatrix},&
&F_2 = \begin{pmatrix}0& 0& 0\\ 0& 0& 0\\ 0& 1& 0\end{pmatrix},&
&F_3 = \begin{pmatrix}0& 0& 0\\ 0& 0& 0\\ 1& 0& 0\end{pmatrix},&
&H_2 = \begin{pmatrix}0& 0& 0\\ 0& 1& 0\\ 0& 0& -1\end{pmatrix}&
\end{align*}
with the commutators
\begin{align*}
&[E_i,F_j]=\delta_{ij}H_i,& &[H_i,H_j]=0,& &[H_i,E_i]=2E_i,& &[H_i,F_i]=-2F_i,& \\
&[H_1,E_2]=-E_2,& &[H_2,E_1]=-E_1,& &[H_2,E_3]=E_3,& &[H_1,E_3]=E_3,& \\
&[H_1,F_2]=F_2,& &[H_2,F_1]=F_1,& &[H_2,F_3]=-F_3,& &[H_1,F_3]=-F_3& \\
\end{align*}
We shall use the symmetric bilinear form $\langle x,y\rangle = Tr(xy)$:
\begin{align*}
&\langle E_i,E_j\rangle =0,& &\langle F_i,F_j\rangle =0,& &\langle H_i,E_j\rangle =0,& &\langle H_i,F_j\rangle =0,& \\
&\langle E_i,F_j\rangle =\delta_{ij},& &\langle H_i,H_i\rangle =2,& &\langle H_1,H_2\rangle =-1& & &
\end{align*}
One can easily check that it is ad-invariant and non-degenerate. The corresponding dual basis is
\[
	H_1^*=\frac{2}{3}H_1+\frac{1}{3}H_2, H_2^*=\frac{1}{3}H_1+\frac{2}{3}H_2, E_i^*=F_i, F_i^*=E_i
\]
and, hence
\begin{align*}
\phi_{\SL_3}(\begin{tikzpicture}[baseline={([yshift=-.5ex]current bounding box.center)},scale=0.3]
	\draw (0,0) circle (1);
	\draw (1,0) -- (-1,0);
	\fill[black] (1,0) circle (3pt)
                 (-1,0) circle (3pt);
\end{tikzpicture})&=c_2=\sum_i e_ie_i^* =\frac{2}{3}H_1^2+\frac{2}{3}H_2^2+\frac{1}{3}(H_1H_2+H_2H_1)+\sum_{i=1}^{3}(E_iF_i+F_iE_i) \\
\phi^{St}_{\SL_3}(\begin{tikzpicture}[baseline={([yshift=-.5ex]current bounding box.center)},scale=0.3]
	\draw (0,0) circle (1);
	\draw (1,0) -- (-1,0);
	\fill[black] (1,0) circle (3pt)
                 (-1,0) circle (3pt);
\end{tikzpicture})&=Tr(\frac{8}{3}\times id_3)=8
\end{align*}
In addition, 
\begin{align*}
\phi^{St}_{\SL_3}(\begin{tikzpicture}[baseline={([yshift=-.5ex]current bounding box.center)},scale=0.3]
	\draw (0,0) circle (1);
	\draw ( 45:1) -- (225:1);
	\draw (135:1) -- (315:1);
	\fill[black] ( 45:1)  circle (3pt)
                 (135:1)  circle (3pt)
                 (225:1)  circle (3pt)
                 (315:1)  circle (3pt);
\end{tikzpicture})&=Tr(\sum_i e_ie_je_i^*e_J^*) =Tr(-\frac{8}{9}\times id_3)=-\frac{8}{3} \end{align*}
Indeed,
\begin{align*}
\phi_{\SL_3}(\begin{tikzpicture}[baseline={([yshift=-.5ex]current bounding box.center)},scale=0.3]
	\draw (0,0) circle (1);
	\draw ( 45:1) -- (225:1);
	\draw (135:1) -- (315:1);
	\fill[black] ( 45:1)  circle (3pt)
                 (135:1)  circle (3pt)
                 (225:1)  circle (3pt)
                 (315:1)  circle (3pt);
\end{tikzpicture})&=(c_2-\lambda)c_2,
\end{align*}
and we have $(\frac{8}{3}-\lambda)\frac{8}{3}=-\frac{8}{9}$, then $\lambda=3$. So
$\phi_{\SL_3}(\begin{tikzpicture}[baseline={([yshift=-.5ex]current bounding box.center)},scale=0.3]
	\draw (0,0) circle (1);
	\draw ( 45:1) -- (225:1);
	\draw (135:1) -- (315:1);
	\fill[black] ( 45:1)  circle (3pt)
                 (135:1)  circle (3pt)
                 (225:1)  circle (3pt)
                 (315:1)  circle (3pt);
\end{tikzpicture})=(c_2-3)c_2$.

Because of the remark in the end of the previous section,  the degree of a Lie algebra weight system of chord diagram with $n$ chords is at most $2n$. By computing the results 
for sufficiently many irreducible representations, 
we can reconstruct the value of the universal weight system. 
Here are some results for the Lie algebra $\SL_3$ in small orders.
\[ 
\begin{array}{c||c|r|r}
  &\thead{\text{intersection} \\\text{graph}} &\SL_3 \text{ weight system \ \ \ \ \ \ \ \ \ \ \ \ \ \ }& \thead{\SL_3\text{ weight system of }\\\text{ the projection in} \\\text{ primitive space}} \\ \hline
\makecell{\gape{\begin{tikzpicture}
	\draw (0,0) circle (6pt);
	\draw (0:6pt) -- (180:6pt);
	\fill[black] (0:6pt) circle (1pt)
	           (180:6pt) circle (1pt);
\end{tikzpicture}}}&K_1& c_2 & c_2  \\ \hline
\makecell{\gape{\begin{tikzpicture}
	\draw (0,0) circle (6pt);
	\draw ( 45:6pt) -- (225:6pt)
          (135:6pt) -- (315:6pt);
	\fill[black] ( 45:6pt) circle (1pt)
                 (135:6pt) circle (1pt)
                 (225:6pt) circle (1pt)
                 (315:6pt) circle (1pt);
\end{tikzpicture}}}&K_{2}& c_2(c_2-3) &  -3c_2  \\ \hline
\makecell{\gape{\begin{tikzpicture}
	\draw (0,0) circle (6pt);
	\draw (  0:6pt) -- (180:6pt)
          ( 60:6pt) -- (300:6pt)
          (120:6pt) -- (240:6pt);
	\fill[black] (  0:6pt) circle (1pt)
                 ( 60:6pt) circle (1pt)
                 (120:6pt) circle (1pt)
                 (180:6pt) circle (1pt)
                 (240:6pt) circle (1pt)
                 (300:6pt) circle (1pt);
\end{tikzpicture}}}&K_{1,2}& c_2(c_2-3)^2 & 9c_2  \\ \hline
\makecell{\gape{\begin{tikzpicture}
	\draw (0,0) circle (6pt);
	\draw (  0:6pt) -- (180:6pt)
          ( 60:6pt) -- (240:6pt)
          (120:6pt) -- (300:6pt);
	\fill[black] (  0:6pt) circle (1pt)
                 ( 60:6pt) circle (1pt)
                 (120:6pt) circle (1pt)
                 (180:6pt) circle (1pt)
                 (240:6pt) circle (1pt)
                 (300:6pt) circle (1pt);
\end{tikzpicture}}}&K_{3}& c_2(c_2-3)(c_2-6) & 18c_2  \\ \hline
\makecell{\gape{\begin{tikzpicture}
	\draw (0,0) circle (6pt);
	\draw ( 20:6pt) -- (160:6pt)
          ( 70:6pt) -- (290:6pt)
          (110:6pt) -- (250:6pt)
          (200:6pt) -- (340:6pt);
	\fill[black] ( 20:6pt) circle (1pt)
                 ( 70:6pt) circle (1pt)
                 (110:6pt) circle (1pt)
                 (160:6pt) circle (1pt)
                 (200:6pt) circle (1pt)
                 (250:6pt) circle (1pt)
                 (290:6pt) circle (1pt)
                 (340:6pt) circle (1pt);
\end{tikzpicture}}}&K_{2,2}& c_2(c_2^3-12c_2^2+63c_2-99) &  c_2(9c_2-99) \\ \hline
\makecell{\gape{\begin{tikzpicture}
	\draw (0,0) circle (6pt);
	\draw ( 20:6pt) -- (160:6pt)
          ( 70:6pt) -- (250:6pt)
          (110:6pt) -- (290:6pt)
          (200:6pt) -- (340:6pt);
	\fill[black] ( 20:6pt) circle (1pt)
                 ( 70:6pt) circle (1pt)
                 (110:6pt) circle (1pt)
                 (160:6pt) circle (1pt)
                 (200:6pt) circle (1pt)
                 (250:6pt) circle (1pt)
                 (290:6pt) circle (1pt)
                 (340:6pt) circle (1pt);
\end{tikzpicture}}}&K_{4}\backslash \{e\}& c_2(c_2^3-15c_2^2+81c_2-126) &  c_2(9c_2-126) \\ \hline
\makecell{\gape{\begin{tikzpicture}
	\draw (0,0) circle (6pt);
	\draw (  0:6pt) -- (180:6pt)
          ( 45:6pt) -- (225:6pt)
          ( 90:6pt) -- (270:6pt)
          (135:6pt) -- (315:6pt);
	\fill[black] (  0:6pt) circle (1pt)
                 ( 45:6pt) circle (1pt)
                 ( 90:6pt) circle (1pt)
                 (135:6pt) circle (1pt)
                 (180:6pt) circle (1pt)
                 (225:6pt) circle (1pt)
                 (270:6pt) circle (1pt)
                 (315:6pt) circle (1pt);
\end{tikzpicture}}}&K_{4}& c_2(c_2^3-18c_2^2+108c_2-180) & c_2(9c_2-180)   \\ \hline
\makecell{\gape{\begin{tikzpicture}
	\draw (0,0) circle (6pt);
	\draw ( 20:6pt) -- (160:6pt)
          ( 60:6pt) -- (300:6pt)
          (120:6pt) -- (240:6pt)
          (200:6pt) -- (340:6pt)
          ( 90:6pt) -- (270:6pt);
	\fill[black] ( 20:6pt) circle (1pt)
                 ( 60:6pt) circle (1pt)
                 (120:6pt) circle (1pt)
                 (160:6pt) circle (1pt)
                 (200:6pt) circle (1pt)
                 (240:6pt) circle (1pt)
                 (300:6pt) circle (1pt)
                 (340:6pt) circle (1pt)
                 ( 90:6pt) circle (1pt)
                 (270:6pt) circle (1pt);
\end{tikzpicture}}}&K_{2,3}& c_2(c_2^4-18c_2^3+162c_2^2-648c_2+873) & -c_2(135c_2-873)   \\ \hline
\makecell{\gape{\begin{tikzpicture}
	\draw (0,0) circle (6pt);
	\draw ( 36:6pt) -- (36+180:6pt)
          ( 72:6pt) -- (72+180:6pt)
          (108:6pt) -- (108+180:6pt)
          (144:6pt) -- (144+180:6pt)
          (  0:6pt) -- (180:6pt);
	\fill[black] (  0:6pt) circle (1pt)
                 ( 36:6pt) circle (1pt)
                 ( 72:6pt) circle (1pt)
                 (108:6pt) circle (1pt)
                 (144:6pt) circle (1pt)
                 (180:6pt) circle (1pt)
                 (216:6pt) circle (1pt)
                 (252:6pt) circle (1pt)
                 (288:6pt) circle (1pt)
                 (324:6pt) circle (1pt);
\end{tikzpicture}}}&K_{5}& c_2(c_2^4-30c_2^3+360c_2^2-1764c_2+2664) & -c_2(324c_2-2664)   \\ \hline
\makecell{\gape{\begin{tikzpicture}
	\draw (0,0) circle (6pt);
	\draw ( 15:6pt) -- (165:6pt)
          ( 45:6pt) -- (-45:6pt)
          ( 75:6pt) -- (-75:6pt)
          (105:6pt) -- (-105:6pt)
          (135:6pt) -- (-135:6pt)
          (-15:6pt) -- (195:6pt);
	\fill[black] ( 15:6pt) circle (1pt)
                 ( 45:6pt) circle (1pt)
                 ( 75:6pt) circle (1pt)
                 (105:6pt) circle (1pt)
                 (135:6pt) circle (1pt)
                 (165:6pt) circle (1pt)
                 (195:6pt) circle (1pt)
                 (-15:6pt) circle (1pt)
                 (-45:6pt) circle (1pt)
                 (-75:6pt) circle (1pt)
                 (-105:6pt) circle (1pt)
                 (-135:6pt) circle (1pt);
\end{tikzpicture}}}&K_{2,4}& \makecell[r]{c_2(c_2^5 - 24c_2^4 + 306c_2^3 - \\ 1998c_2^2 + 6273c_2 - 7227)} & c_2(1485c_2-7227) \\ \hline
\makecell{\gape{\begin{tikzpicture}
	\draw (0,0) circle (6pt);
	\draw (  0:6pt) -- (180:6pt)
          ( 30:6pt) -- (150:6pt)
          ( 60:6pt) -- (300:6pt)
          ( 90:6pt) -- (270:6pt)
          (120:6pt) -- (240:6pt)
          (210:6pt) -- (330:6pt);
	\fill[black] (  0:6pt) circle (1pt)
                 ( 30:6pt) circle (1pt)
                 ( 60:6pt) circle (1pt)
                 ( 90:6pt) circle (1pt)
                 (120:6pt) circle (1pt)
                 (150:6pt) circle (1pt)
                 (180:6pt) circle (1pt)
                 (210:6pt) circle (1pt)
                 (240:6pt) circle (1pt)
                 (270:6pt) circle (1pt)
                 (300:6pt) circle (1pt)
                 (330:6pt) circle (1pt);
\end{tikzpicture}}}&K_{3,3}&\makecell[r]{\frac{3}{4}c_3^2+ c_2(c_2^5 - 27c_2^4 +\\ 405c_2^3 - 3339c_2^2 + 13014c_2 - 17595)} & \makecell[r]{\frac{3}{4}c_3^2-c_2(99c_2^2- \\ 4374c_2+17595)} \\ \hline
\makecell{\gape{\begin{tikzpicture}
	\draw (0,0) circle (6pt);
	\draw (  0:6pt) -- (180:6pt)
          ( 30:6pt) -- (210:6pt)
          ( 60:6pt) -- (240:6pt)
          ( 90:6pt) -- (270:6pt)
          (120:6pt) -- (300:6pt)
          (150:6pt) -- (330:6pt);
	\fill[black] (  0:6pt) circle (1pt)
                 ( 30:6pt) circle (1pt)
                 ( 60:6pt) circle (1pt)
                 ( 90:6pt) circle (1pt)
                 (120:6pt) circle (1pt)
                 (150:6pt) circle (1pt)
                 (180:6pt) circle (1pt)
                 (210:6pt) circle (1pt)
                 (240:6pt) circle (1pt)
                 (270:6pt) circle (1pt)
                 (300:6pt) circle (1pt)
                 (330:6pt) circle (1pt);
\end{tikzpicture}}}&K_{6}&\makecell[r]{ c_3^2 +c_2(c_2^5 - 45c_2^4 + \\ 900c_2^3 - 8826c_2^2 + 38196c_2 - 54288)}& \makecell[r]{c_3^2-c_2(132c_2^2- \\ 10872c_2+54288)} \\ \hline
\end{array}
\]

Here $c_2$, $c_3$ are the Casimir elements,
of degree~$2$ and~$3$, respectively,
 in $ZU(\SL_3)$,
\begin{align*}
c_2 =& \sum_{i=1}^8 e_i e_i^*, \ \ \ \ \ \{e_i\}_{i\in [8]} \text{\  is a basis of\ }\SL_3. \\
c_3 =&-\frac{4}{3}\Big(2H_1H_1H_1-3H_1H_1(H_1 + H_2)-3H_1(H_1 + H_2)(H_1 + H_2) +\\&+ 2(H_1 + H_2)(H_1 + H_2)(H_1 + H_2)+  9E_1F_1H_1-18E_1F_1(H_1 + H_2)-\\& -18E_3F_3H_1 +  9E_3F_3(H_1 + H_2) + 9E_2F_2H_1 +  9E_2F_2(H_1 + H_2)-\\ &-27E_1F_3E_2-27F_1E_3F_2 + 18H_1(H_1 + H_2)-9(H_1 + H_2)(H_1 + H_2)-\\&-18H_1 + 9(H_1 + H_2)\Big)
\end{align*}
The factor $-\frac{4}{3}$ in $c_3$
is due to the requirement that under the standard representation $c_3$ takes value $\frac{80}{3}$, which is the image of the Casimir element $\sum_{i,j,k}e_{ij} e_{jk} e_{ki} - \frac{1}{3}(\sum_{i}e_{ii})^3$ in $ZU(gl_3)$ .

\section{Jacobi diagrams and Lie algebra weight systems}\label{s4}
When computing the values of the $\SL_3$ weight system, we will
require the results in~\cite{yoshizumi1999some} about recurrence
relations for the values of this weight system on Jacobi diagrams.
To this end, we recall the notion of {\it closed Jacobi diagram}. These diagrams provide a better understanding of the primitive space $P\cA$,
see, e.g.~\cite{Chmutov2012}.

\begin{definition}
A {\it closed Jacobi diagram} (or, simply, a {\it closed diagram}) is a connected trivalent graph with a distinguished embedded oriented cycle, called Wilson loop, and a fixed cyclic order of half-edges at each vertex not on the Wilson loop. Half the number of the vertices of a closed diagram is called the degree, or order, of the diagram. This number is always an integer.
\end{definition}

 In the pictures below, we shall always draw the diagram inside its Wilson loop, which will be assumed to be oriented counterclockwise unless explicitly specified otherwise. Inner vertices will also be assumed to be oriented counterclockwise.

Chord diagrams are exactly those closed Jacobi diagrams all of whose vertices lie on the Wilson loop.

\begin{definition}
The vector space of closed diagrams $\cC_n^{STU}$ is the space spanned by all closed diagrams $\cC_n$ of degree $n$ modulo the {\it STU relations}:
\[
\begin{tikzpicture}[baseline={([yshift=-.5ex]current bounding box.center)}]
	\draw[->][line width=1pt]  ([shift=(230:1cm)]0,0) arc [start angle=230, end angle=310, radius=1];
	\draw (270:0.3)--(270:1);
	\draw (270:0.3)--(  0:0.3);
	\draw (270:0.3)--(180:0.3);
	\fill [black] (270:0.3) circle (1pt);
	\node at (1,-1) {S};
\end{tikzpicture} =
\begin{tikzpicture}[baseline={([yshift=-.5ex]current bounding box.center)}]
	\draw[->][line width=1pt]  ([shift=(230:1cm)]0,0) arc [start angle=230, end angle=310, radius=1];
	\draw (  0:0.3) ..  controls (-60:0.3) and (-100:0.5)  .. (280:1);
	\draw (180:0.3) ..  controls (240:0.3) and (280:0.5)  .. (260:1);
	\node at (1,-1) {T};
\end{tikzpicture} 	-
\begin{tikzpicture}[baseline={([yshift=-.5ex]current bounding box.center)}]
	\draw[->][line width=1pt]  ([shift=(230:1cm)]0,0) arc [start angle=230, end angle=310, radius=1];
	\draw (  0:0.3) -- (260:1);
	\draw (180:0.3) -- (280:1);
	\node at (1,-1) {U};
\end{tikzpicture} 	
\]
The three diagrams $S$, $T$ and $U$ must be identical outside the shown fragment. We write $\cC^{STU}$ for the direct sum of the spaces $\cC_n^{STU}$ for all $n \ge 0$.
\end{definition}

Now we shall define a bialgebra structure in the space $\cC^{STU}$.

\begin{definition}
The product of two closed diagrams is defined in the same way as for chord diagrams: the two Wilson loops are cut at arbitrary places and then glued together into one loop, in agreement with the orientations:
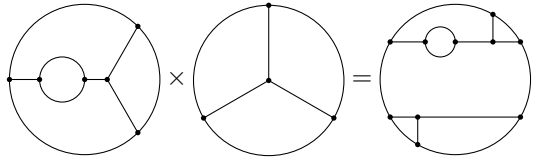
\[
\begin{tikzpicture}[baseline={([yshift=-.5ex]current bounding box.center)}]
	\draw (0,0) circle (1)
		  (-0.3,0) circle (0.3);
	\draw (-1, 0) -- (-0.6,-0)
          (0,-0) -- (0.3,-0)
          (0.3,-0) -- (45:1)
          (0.3,-0) -- (-45:1);
	\fill[black] ( 0.3,-0) circle (1pt)
                 (-0.6,-0) circle (1pt)
                 ( 0, 0) circle (1pt)
                 (45:1) circle (1pt)
                 (-45:1) circle (1pt)
                 (-1,-0) circle (1pt);
\end{tikzpicture} \times
\begin{tikzpicture}[baseline={([yshift=-.5ex]current bounding box.center)}]
	\draw (0,0) circle (1);
	\draw ( 0, 0) -- (-30:1)
          ( 0, 0) -- (90:1)
          ( 0, 0) -- (210:1);
	\fill[black] ( 210:1) circle (1pt)
                 ( 0, 0) circle (1pt)
                 (90:1) circle (1pt)
                 (-30:1) circle (1pt);
\end{tikzpicture} = 
\begin{tikzpicture}[baseline={([yshift=-.5ex]current bounding box.center)}]
	\draw (0,0) circle (1)
		  (-0.2,0.5) circle (0.2);
	\draw (150:1) -- (-0.4,0.5)
          (30:1) -- (0,0.5)
          (60:1) -- (0.5,0.5)
          (210:1) -- (-30:1)
          (240:1) -- (-0.5,-0.5);
	\fill[black] (150:1)circle (1pt)
                 (30:1) circle (1pt)
                 (60:1) circle (1pt)
                 (210:1)circle (1pt)
                 (240:1)circle (1pt)
                  (-0.4,0.5)circle (1pt)
                 (0,0.5)circle (1pt)
                 (0.5,0.5)circle (1pt)
                 (-30:1)circle (1pt)
                  (-0.5,-0.5)circle (1pt);
\end{tikzpicture} 
\]
\end{definition}

\begin{definition}
The {\it internal} graph of a closed diagram is the graph obtained by erasing the Wilson loop. A closed diagram is said to be {\it connected} if its internal graph is  connected. The {\it connected components} of a closed diagram are defined as the connected components of its internal graph.
\end{definition}

In the sense of this definition, any chord diagram of order $n$ consists of $n$ connected components — the maximal possible number.

Now, the construction of the coproduct proceeds in the same way as for chord diagrams, the chords being replaced by the
more general connected components.

\begin{definition}
Let $D$ be a closed diagram and $[D]$ the set of its connected components. For any subset $J \subseteq [D]$, denote by $D_J$ the closed diagram with only those components that belong to $J$ and by $D_{\bar{J}}$ the “complementary” diagram $(\bar{J} := [D] \setminus J)$. We set
\[
	\delta(D):=\sum_{J\subseteq[D]} D_J\otimes D_{\bar{J}}
\]
\end{definition}

Now, for each~$n=0,1,2,\dots$, we have a natural inclusion $\lambda : \cA_n \to \cC_n $.

\begin{claim}~\cite{bar1995vassiliev}
The inclusion $\lambda$ gives rise to an isomorphism of bialgebras $\lambda : \cA^{fr} \to \cC^{STU}$. 
\end{claim}

By definition, connected closed diagrams are primitive with respect to the coproduct $\delta$. It may sound surprising that the converse is also true:
\begin{claim}\cite{bar1995vassiliev}
The primitive space $P$ of the bialgebra $\cC^{STU}$ coincides with the linear span of connected closed diagrams.
\end{claim}

Since every closed diagram is a linear combination of chord diagrams, the weight system $\phi_\fg$ can be treated as a function on $\cC^{STU}$ with values in $U(\fg)$.

The STU relation, which defines the algebra $\cC$, gives us a hint how to do it. Namely, if we assign elements $e_i , e_j$ to the endpoints of chords of the T- and U- diagrams from the STU relations,
\[
\begin{tikzpicture}[baseline={([yshift=-.5ex]current bounding box.center)}]
	\draw[->][line width=1pt]  ([shift=(230:1cm)]0,0) arc [start angle=230, end angle=310, radius=1];
	\draw (  0:0.3) ..  controls (-60:0.3) and (-100:0.5)  .. (280:1);
	\draw (180:0.3) ..  controls (240:0.3) and (280:0.5)  .. (260:1);
	\node at (1,-1) {T};
	\node [below] at ( 0.3,-1) {$e_j$};
	\node [below] at (-0.3,-1) {$e_i$};
	\node [above] at (  0:0.3) {$e_j^*$};
	\node [above] at (180:0.3) {$e_i^*$};
\end{tikzpicture} 	-
\begin{tikzpicture}[baseline={([yshift=-.5ex]current bounding box.center)}]
	\draw[->][line width=1pt]  ([shift=(230:1cm)]0,0) arc [start angle=230, end angle=310, radius=1];
	\draw (  0:0.3) -- (260:1);
	\draw (180:0.3) -- (280:1);
	\node at (1,-1) {U};
	\node [below] at ( 0.3,-1) {$e_i^*$};
	\node [below] at (-0.3,-1) {$e_j^*$};
	\node [above] at (  0:0.3) {$e_j^*$};
	\node [above] at (180:0.3) {$e_i^*$};
\end{tikzpicture}=
\begin{tikzpicture}[baseline={([yshift=-.5ex]current bounding box.center)}]
	\draw[->][line width=1pt]  ([shift=(230:1cm)]0,0) arc [start angle=230, end angle=310, radius=1];
	\draw (270:0.3)--(270:1);
	\draw (270:0.3)--(  0:0.3);
	\draw (270:0.3)--(180:0.3);
	\fill [black] (270:0.3) circle (1pt);
	\node at (1,-1) {S};
	\node [below] at ( 0,-1) {$[e_i,e_j]$};
	\node [above] at (  0:0.3) {$e_j^*$};
	\node [above] at (180:0.3) {$e_i^*$};
\end{tikzpicture}	
\]
then it is natural to assign the commutator $[e_i,e_j]$ to the trivalent vertex on the Wilson loop of the S-diagram.

Generally, $[e_i,e_j]$ may not be a basis vector. A diagram with an endpoint marked by a linear combination of the basis vectors should be understood as the corresponding linear combination of diagrams marked by basis vectors. This understanding implies a useful
\begin{lemma}
The degree of the value of a Lie algebra weight system on a closed diagram $D$ is less or equal than the number of legs of $D$.
\end{lemma}

\section{Values of the $\SL_3$ weight system on certain Jacobi diagrams }\label{s5}
Given a weight system $w$, we write $\bar{w}:=w\circ\pi$ 
for its composition with the projection to the subspace of primitives
along the subspace of decomposable elements.

Here is an important lemma.
\begin{lemma}[leaf lemma]~\cite{Chmutov2007} Let $w_{\fg}$ be the weight system associated to a  metrized simple Lie algebra $\fg$ where the metric is proportional
to the Killing form, with proportionality coefficient~$\lambda$, 
	$c$ the quadratic Casimir element
	in $U(\fg)$. Then 
\begin{align*}
	w_{\fg}(D)=\left(c-\frac{1}{2\lambda}\right)w_{\fg}(D_a)
\end{align*}
for any Jacobi diagram $D$ and a leaf~$a$ (a chord intersecting 
a single leg) in it, $D_a$ being the Jacobi diagram $D$ with the chord $a$ removed.
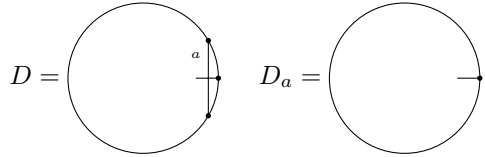
\begin{align*}
	D=\begin{tikzpicture}[baseline={([yshift=-.5ex]current bounding box.center)}]
	\draw (0,0) circle (1);
	\draw ( 30:1) -- (-30:1)
	      (0.7,0) -- (1,0);
	\fill[black] ( 30:1) circle (1pt)
                 (-30:1) circle (1pt)
                 (1,0) circle (1pt);
    \node [scale=1] at ( 0.7, 0.3) {\tiny $a$};
\end{tikzpicture}\ \ \ \ 
	D_a=\begin{tikzpicture}[baseline={([yshift=-.5ex]current bounding box.center)}]
	\draw (0,0) circle (1);
	\draw 
	      (0.7,0) -- (1,0);
	\fill[black] 
                 (1,0) circle (1pt);
\end{tikzpicture}
\end{align*}
\end{lemma}

For small degree Jacobi diagrams, we can resolve all the vertices and get the linear combination of chord diagram. For example
\[
	\begin{tikzpicture}[baseline={([yshift=-.5ex]current bounding box.center)},scale=0.3]
	\draw (0,0) circle (1);
	\draw ( 30:1) -- ( 0, 0.2)
          (-30:1) -- ( 0,-0.2)
          (150:1) -- (-0, 0.2)
          (210:1) -- (-0,-0.2)
          (-0,0.2) -- (-0,-0.2);
	\fill[black] ( 0, 0.2) circle (3pt)
                 (-0,-0.2) circle (3pt)
                 ( 30:1) circle (3pt)
                 (-30:1) circle (3pt)
                 (150:1) circle (3pt)
                 (210:1) circle (3pt);
	\end{tikzpicture}=
\begin{tikzpicture}[baseline={([yshift=-.5ex]current bounding box.center)},scale=0.3]
	\draw (0,0) circle (1);
	\draw (1,0) -- (-1,0);
	\fill[black] (1,0) circle (3pt)
                 (-1,0) circle (3pt)
                 (60:1) circle (3pt)
                 (120:1) circle (3pt)
                 (240:1) circle (3pt)
                 (300:1) circle (3pt);
	\draw (60:1)  ..  controls (0,0.2) and (0,-0.2)  .. (300:1);
	\draw (120:1) ..  controls (0,0.2) and (0,-0.2)  .. (240:1);
\end{tikzpicture}-2
\begin{tikzpicture}[baseline={([yshift=-.5ex]current bounding box.center)},scale=0.3]
	\draw (0,0) circle (1);
	\fill[black] (0:1) circle (3pt)
                 (180:1) circle (3pt)
                 (60:1) circle (3pt)
                 (120:1) circle (3pt)
                 (210:1) circle (3pt)
                 (330:1) circle (3pt);
	\draw (60:1)--(180:1);
	\draw (0:1)--(120:1);
	\draw (210:1)--(330:1);
\end{tikzpicture}+
\begin{tikzpicture}[baseline={([yshift=-.5ex]current bounding box.center)},scale=0.3]
	\draw (0,0) circle (1);
	\draw (1,0) -- (-1,0);
	\fill[black] (1,0) circle (3pt)
                 (-1,0) circle (3pt)
                 (45:1) circle (3pt)
                 (135:1) circle (3pt)
                 (225:1) circle (3pt)
                 (315:1) circle (3pt);
	\draw (45:1)--(135:1);
	\draw (225:1)--(315:1);
\end{tikzpicture}
\]
Based on the above table, we have 
\[
	w_{\SL_3}(\begin{tikzpicture}[baseline={([yshift=-.5ex]current bounding box.center)},scale=0.3]
	\draw (0,0) circle (1);
	\draw ( 30:1) -- ( 0, 0.2)
          (-30:1) -- ( 0,-0.2)
          (150:1) -- (-0, 0.2)
          (210:1) -- (-0,-0.2)
          (-0,0.2) -- (-0,-0.2);
	\fill[black] ( 0, 0.2) circle (3pt)
                 (-0,-0.2) circle (3pt)
                 ( 30:1) circle (3pt)
                 (-30:1) circle (3pt)
                 (150:1) circle (3pt)
                 (210:1) circle (3pt);
	\end{tikzpicture})=c_2(c_2-3)^2-2c_2^2(c_2-2)+c_2^3=9c_2
\]
Here are the results of similar computations for some small Jacobi diagrams
\[
\begin{array}{c|r}
\text{Jacobi \  diagram}&\SL_3 \text{ weight system } \\ \hline
\makecell{\gape{\begin{tikzpicture}[baseline={([yshift=-.5ex]current bounding box.center)},scale=0.3]
	\draw (0,0) circle (1);
	\draw ( 30:1) -- ( 0, 0.2)
          (-30:1) -- ( 0,-0.2)
          (150:1) -- (-0, 0.2)
          (210:1) -- (-0,-0.2)
          (-0,0.2) -- (-0,-0.2);
	\fill[black] ( 0, 0.2) circle (3pt)
                 (-0,-0.2) circle (3pt)
                 ( 30:1) circle (3pt)
                 (-30:1) circle (3pt)
                 (150:1) circle (3pt)
                 (210:1) circle (3pt);
	\end{tikzpicture}}}& 9c_2  \\ \hline
\makecell{\gape{\begin{tikzpicture}[baseline={([yshift=-.5ex]current bounding box.center)},scale=0.3]
	\draw (0,0) circle (1);
	\draw ( 30:1) -- ( 0.2, 0.2)
          (-30:1) -- ( 0.2,-0.2)
          (150:1) -- (-0.2, 0.2)
          (210:1) -- (-0.2,-0.2)
          (0.2,0.2) -- (0.2,-0.2)
          (0.2,0.2) -- (-0.2,0.2)
          (-0.2,-0.2) -- (0.2,-0.2)
          (-0.2,-0.2) -- (-0.2,0.2);
	\fill[black] ( 0.2, 0.2) circle (3pt)
                 (-0.2,-0.2) circle (3pt)
                 (-0.2,0.2) circle (3pt)
                 (0.2,-0.2) circle (3pt)
                 ( 30:1) circle (3pt)
                 (-30:1) circle (3pt)
                 (150:1) circle (3pt)
                 (210:1) circle (3pt);
	\end{tikzpicture}}}&c_2(9c_2+9)  \\ \hline
\makecell{\gape{\begin{tikzpicture}[baseline={([yshift=-.5ex]current bounding box.center)},scale=0.3]
	\draw (0,0) circle (1);
	\draw ( 30:1) -- ( 0.4, 0.2)
          (-30:1) -- ( 0.4,-0.2)
          (150:1) -- (-0.4, 0.2)
          (210:1) -- (-0.4,-0.2)
          (-0.4,0.2) -- (-0.4,-0.2)
          (-0,0.2) -- (-0,-0.2)
          (0.4,0.2) -- (0.4,-0.2)
          (-0.4,0.2) -- (0.4,0.2)
          (-0.4,-0.2) -- (0.4,-0.2);
	\fill[black] ( 0, 0.2) circle (3pt)
                 (-0,-0.2) circle (3pt)
                 (-0.4,-0.2) circle (3pt)
                 (0.4,-0.2) circle (3pt)
                 (-0.4,0.2) circle (3pt)
                 (0.4,0.2) circle (3pt)
                 ( 30:1) circle (3pt)
                 (-30:1) circle (3pt)
                 (150:1) circle (3pt)
                 (210:1) circle (3pt);
	\end{tikzpicture}}}&c_2(27c_2+63)  \\ \hline
\makecell{\gape{\begin{tikzpicture}[baseline={([yshift=-.5ex]current bounding box.center)},scale=0.3]
	\draw (0,0) circle (1);
	\draw ( 30:1) -- ( 0.6, 0.2)
          (-30:1) -- ( 0.6,-0.2)
          (150:1) -- (-0.6, 0.2)
          (210:1) -- (-0.6,-0.2)
          (0.6,0.2) -- (0.6,-0.2)
          (-0.6,0.2) -- (-0.6,-0.2)
          (0.2,0.2) -- (0.2,-0.2)
          (-0.2,0.2) -- (-0.2,-0.2)
          (-0.6,-0.2) -- (0.6,-0.2)
          (-0.6,0.2) -- (0.6,0.2);
	\fill[black] ( 0.2, 0.2) circle (3pt)
                 (-0.2,-0.2) circle (3pt)
                 (-0.2,0.2) circle (3pt)
                 (0.2,-0.2) circle (3pt)
                 (0.6,-0.2) circle (3pt)
                 (0.6,0.2) circle (3pt)
                 (-0.6,-0.2) circle (3pt)
                 (-0.6,0.2) circle (3pt)
                 ( 30:1) circle (3pt)
                 (-30:1) circle (3pt)
                 (150:1) circle (3pt)
                 (210:1) circle (3pt);
	\end{tikzpicture}}}&c_2(189c_2+117)  \\ \hline
\end{array}
\]

Denote by $J_{i,j}$ the order $i+j+2$ Jacobi diagram with $i-1$ cells and $j$ chords crossing cells. ($i,j\ge 0$)
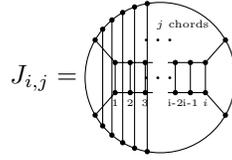
\[
	J_{i,j}=\begin{tikzpicture}[baseline={([yshift=-.5ex]current bounding box.center)}]
	\draw (0,0) circle (1);
	\draw ( 30:1) -- ( 0.6, 0.2)
          (-30:1) -- ( 0.6,-0.2)
          (150:1) -- (-0.6, 0.2)
          (210:1) -- (-0.6,-0.2)

          (180+40:1) -- (180-40:1)
          (180+50:1) -- (180-50:1)
          (180+60:1) -- (180-60:1)
          (180+70:1) -- (180-70:1)
          (180+80:1) -- (180-80:1)
          ( 0.6, 0.2) -- ( 0.6,-0.2)
          (-0.6, 0.2) -- (-0.6,-0.2)
          ( 0.6, 0.2) -- ( 0.1, 0.2)
          ( 0.6,-0.2) -- ( 0.1,-0.2)
          (-0.6, 0.2) -- (-0.1, 0.2)
          (-0.6,-0.2) -- (-0.1,-0.2)
          ( 0.4, 0.2) -- ( 0.4,-0.2)
          (-0.4, 0.2) -- (-0.4,-0.2)
          ( 0.2, 0.2) -- ( 0.2,-0.2)
          (-0.2, 0.2) -- (-0.2,-0.2);
	\fill[black] ( 0.6, 0.2) circle (1pt)
                 (-0.6, 0.2) circle (1pt)
                 ( 0.6,-0.2) circle (1pt)
                 (-0.6,-0.2) circle (1pt)
                 ( 0.4, 0.2) circle (1pt)
                 (-0.4, 0.2) circle (1pt)
                 ( 0.4,-0.2) circle (1pt)
                 (-0.4,-0.2) circle (1pt)
                 ( 0.2, 0.2) circle (1pt)
                 (-0.2, 0.2) circle (1pt)
                 ( 0.2,-0.2) circle (1pt)
                 (-0.2,-0.2) circle (1pt)
                 ( 30:1) circle (1pt)
                 (-30:1) circle (1pt)
                 (150:1) circle (1pt)
                 (210:1) circle (1pt)

                 (100:1) circle (1pt)
                 (110:1) circle (1pt)
                 (120:1) circle (1pt)
                 (130:1) circle (1pt)
                 (140:1) circle (1pt)
                 (220:1) circle (1pt)
                 (230:1) circle (1pt)
                 (240:1) circle (1pt)
                 (250:1) circle (1pt)
                 (260:1) circle (1pt);
    \node at (0,0) {\ldots};
    \node at (0,0.5) {\ldots};    
    \node [below,scale=0.75] at ( 0.6,-0.2) {\tiny $i$};
    \node [below,scale=0.75] at ( 0.4,-0.2) {\tiny i-$1$};
    \node [below,scale=0.75] at ( 0.2,-0.2) {\tiny i-$2$};
    \node [below,scale=0.75] at (-0.2,-0.2) {\tiny $3$};
    \node [below,scale=0.75] at (-0.4,-0.2) {\tiny $2$};
    \node [below,scale=0.75] at (-0.6,-0.2) {\tiny $1$};
    \node [scale=0.8] at (0.3,0.7) {\tiny $j$ chords};
\end{tikzpicture}
\]
Specifically,
\[
	J_{i,0}=
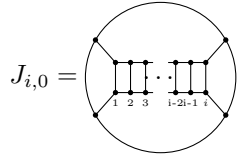
\begin{tikzpicture}[baseline={([yshift=-.5ex]current bounding box.center)}]
	\draw (0,0) circle (1);
	\draw ( 30:1) -- ( 0.6, 0.2)
          (-30:1) -- ( 0.6,-0.2)
          (150:1) -- (-0.6, 0.2)
          (210:1) -- (-0.6,-0.2)
          ( 0.6, 0.2) -- ( 0.6,-0.2)
          (-0.6, 0.2) -- (-0.6,-0.2)
          ( 0.6, 0.2) -- ( 0.1, 0.2)
          ( 0.6,-0.2) -- ( 0.1,-0.2)
          (-0.6, 0.2) -- (-0.1, 0.2)
          (-0.6,-0.2) -- (-0.1,-0.2)
          ( 0.4, 0.2) -- ( 0.4,-0.2)
          (-0.4, 0.2) -- (-0.4,-0.2)
          ( 0.2, 0.2) -- ( 0.2,-0.2)
          (-0.2, 0.2) -- (-0.2,-0.2);
	\fill[black] ( 0.6, 0.2) circle (1pt)
                 (-0.6, 0.2) circle (1pt)
                 ( 0.6,-0.2) circle (1pt)
                 (-0.6,-0.2) circle (1pt)
                 ( 0.4, 0.2) circle (1pt)
                 (-0.4, 0.2) circle (1pt)
                 ( 0.4,-0.2) circle (1pt)
                 (-0.4,-0.2) circle (1pt)
                 ( 0.2, 0.2) circle (1pt)
                 (-0.2, 0.2) circle (1pt)
                 ( 0.2,-0.2) circle (1pt)
                 (-0.2,-0.2) circle (1pt)
                 ( 30:1) circle (1pt)
                 (-30:1) circle (1pt)
                 (150:1) circle (1pt)
                 (210:1) circle (1pt);
    \node at (0,0) {\ldots};
    \node [below,scale=0.75] at ( 0.6,-0.2) {\tiny $i$};
    \node [below,scale=0.75] at ( 0.4,-0.2) {\tiny i-$1$};
    \node [below,scale=0.75] at ( 0.2,-0.2) {\tiny i-$2$};
    \node [below,scale=0.75] at (-0.2,-0.2) {\tiny $3$};
    \node [below,scale=0.75] at (-0.4,-0.2) {\tiny $2$};
    \node [below,scale=0.75] at (-0.6,-0.2) {\tiny $1$};
\end{tikzpicture}
\]
and $J_{0,j}$ is  the chord diagram whose intersection graph is
the complete bipartite graph $K_{2,j}$.

Our first main result is the following 
\begin{theorem}\label{th1}
For any simple Lie algebra~$\fg$ endowed with the scalar product 
proportional to the Killing form with proportionality coefficient~$\lambda$,
one has
\[
	\bar{w}_{\fg}(J_{i,j})=\bar{w}_{\fg}(J_{i-1,j+1})+\frac{1}{\lambda}\bar{w}_{\fg}(J_{i-1,j})
\]
\end{theorem}

Theorem~\ref{th1} implies the following
\begin{corollary}\label{cor1} The following assertions are true:
\begin{enumerate}
\item the value $\bar{w}_{\fg}(J_{0,j})$ has degree at most $4$;
\item we have $\bar{w}_{\fg}(J_{i,0})=\sum_{k=0}^i\tbinom{i}{k}\lambda^{-k}\bar{w}_{\fg}(J_{0,i-k})$;
\item we have $\sum_{n=0} \bar{w}_{\fg}(J_{n,0})\frac{x^n}{n!}=e^{\frac{x}{\lambda}}\sum_{n=0} \bar{w}_{\fg}(J_{0,n})\frac{x^n}{n!}$.
\end{enumerate}

\end{corollary}

\section{Proof of Theorem~\ref{th1}}\label{s6}
In order to prove Theorem~\ref{th1}, we need the \textit{leaf lemma after projection}

\begin{lemma}[leaf lemma after projection]
For any simple Lie algebra~$\fg$ endowed with the scalar product 
proportional to the Killing form with proportionality coefficient~$\lambda$,
one has
\begin{align*}
	\bar{w}_{\fg}(D)=-\frac{1}{2\lambda}\bar{w}_{\fg}(D_a); \ \ \ \  \ \ \ \ \ \ \  \ \\
\end{align*}
for any Jacobi diagram $D$ and any leaf~$a$ in it.
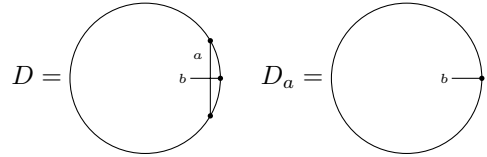
\begin{align*}
	D=\begin{tikzpicture}[baseline={([yshift=-.5ex]current bounding box.center)}]
	\draw (0,0) circle (1);
	\draw ( 30:1) -- (-30:1)
	      (0.6,0) -- (1,0);
	\fill[black] ( 30:1) circle (1pt)
                 (-30:1) circle (1pt)
                 (1,0) circle (1pt);
    \node [scale=1] at ( 0.7, 0.3) {\tiny $a$};
    \node [scale=1] at ( 0.5, 0) {\tiny $b$};
\end{tikzpicture}\ \ \ \ 
	D_a=\begin{tikzpicture}[baseline={([yshift=-.5ex]current bounding box.center)}]
	\draw (0,0) circle (1);
	\draw 
	      (0.6,0) -- (1,0);
	\fill[black] 
                 (1,0) circle (1pt);
    \node [scale=1] at ( 0.5, 0) {\tiny $b$};
\end{tikzpicture}
\end{align*}
\end{lemma}

\begin{proof}
By the formula for the projection, 
\[
	\pi(D)=D-\sum_{i=2}^{|[D]|}(-1)^i(i-1)!\sum_{\substack{\bigsqcup\limits_{j=1}^i [D_j]=[D]\\ [D_j]\ne \emptyset}}\prod_{j=1}^i D_j
\]where the sum is taken over all unordered splittings of the set
 $[D]$ of connected components of $D$ into nonempty subsets $[D_j]$ and $D_j$ is the sub-diagram induced by $[D_j]$. 
Now, after picking a chord $a$, we rewrite the previous formula in the form 
\begin{eqnarray*}
	\pi(D)&=&D-\sum_{i=2}^{|[D]|}(-1)^{i}(i-1)!
	\left(\sum_{\substack{\bigsqcup\limits_{j=1}^i [D_j]=[D]\\ [D_j]\ne \emptyset;\{a\}=D_1}}\{a\}\prod_{j=2}^i D_j+\sum_{\substack{\bigsqcup\limits_{j=1}^i [D_j]=[D]\\ [D_j]\ne \emptyset; \{a,b\}\subseteq D_1}}\prod_{j=1}^i D_j\right. \\
	&&\left.+\sum_{\substack{\bigsqcup\limits_{j=1}^i [D_j]=[D]\\ [D_j]\ne \emptyset;\{a\}\subsetneq D_1;b\notin D_1}}(\{a\}\sqcup D_{a,1})\prod_{j=2}^i D_j\right).
\end{eqnarray*}
Now, applying $w_\fg$ to both sides and using the Leaf lemma we obtain
the required:
\begin{align*}
	\bar{w}_\fg(D)&=(c-\frac{1}{2\lambda})w_\fg(D_a)-\sum_{i=2}^{|[D]|}(-1)^{i}(i-1)!\left(\sum_{\substack{\bigsqcup\limits_{j=1}^{i-1} [D_j]=[D]\\ [D_j]\ne \emptyset}}cw_\fg(\prod_{j=1}^{i-1}D_j)\right. \\
	&\left.+\sum_{\substack{\bigsqcup\limits_{j=1}^i [D_{a,j}]=[D_a]\\ [D_{a,j}]\ne \emptyset}}(c-\frac{1}{2\lambda})w_\fg(\prod_{j=1}^i D_{a,j})\mathrel{\phantom{=}}+\sum_{\substack{\bigsqcup\limits_{j=1}^i [D_{a,j}]=[D_a]\\ [D_{a,j}]\ne \emptyset}}(i-1)cw_\fg(\prod_{j=1}^i D_{a,j})\right) \\
	&=(c-\frac{1}{2\lambda})w_\fg(D_a)-cw_\fg(D_a)\\
	&-\sum_{i=2}^{|[D_a]|}\left((-1)^{i+1}i!+(-1)^{i}(i-1)!(i-1)\right)\sum_{\substack{\bigsqcup\limits_{j=1}^{i} [D_{a,j}]=[D_a]\\ [D_{a,j}]\ne \emptyset}}cw_\fg(\prod_{j=1}^{i}D_{a,j}) \\
	&\mathrel{\phantom{=}}+ (2c-\frac{1}{2\lambda})(\bar{w}_\fg(D_a)-w_\fg(D_a)) \\
	&=(c-\frac{1}{2\lambda})w_\fg(D_a)-cw_\fg(D_a) -c\left(\bar{w}_\fg(D_a)-w_\fg(D_a)\right)\\
	&+ (c-\frac{1}{2\lambda})\left(\bar{w}_\fg(D_a)-w_\fg(D_a)\right) \\
	&= -\frac{1}{2\lambda}\bar{w}_{\fg}(D_a)
\end{align*}
\end{proof}

In order to do the induction step, let us resolve the extreme
internal vertices in $J_{i,j}$ by means of the $STU$ relations:
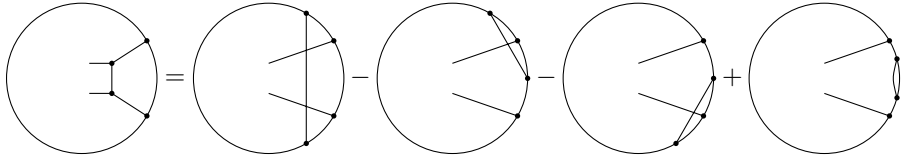
\[
	\begin{tikzpicture}[baseline={([yshift=-.5ex]current bounding box.center)}]
	\draw (0,0) circle (1);
	\draw ( 30:1) -- ( 0.4, 0.2)
          (-30:1) -- ( 0.4,-0.2)
          (0.4,0.2) -- (0.4,-0.2)
          (0.1,-0.2) -- (0.4,-0.2)
          (0.1,0.2) -- (0.4,0.2);
	\fill[black] (0.4,-0.2) circle (1pt)
                 (0.4,0.2) circle (1pt)
                 ( 30:1) circle (1pt)
                 (-30:1) circle (1pt);
	\end{tikzpicture}=
	\begin{tikzpicture}[baseline={([yshift=-.5ex]current bounding box.center)}]
	\draw (0,0) circle (1);
	\draw ( 30:1) -- ( 0, 0.2)
          (-30:1) -- ( 0,-0.2)
          ( 60:1) -- (-60:1);
	\fill[black] ( 30:1) circle (1pt)
                 (-30:1) circle (1pt)
                 (-60:1) circle (1pt)
                 ( 60:1) circle (1pt);
	\end{tikzpicture}-
	\begin{tikzpicture}[baseline={([yshift=-.5ex]current bounding box.center)}]
	\draw (0,0) circle (1);
	\draw ( 30:1) -- ( 0, 0.2)
          (-30:1) -- ( 0,-0.2)
          ( 60:1) -- (  0:1);
	\fill[black] ( 30:1) circle (1pt)
                 (-30:1) circle (1pt)
                 (  0:1) circle (1pt)
                 ( 60:1) circle (1pt);
	\end{tikzpicture}-
	\begin{tikzpicture}[baseline={([yshift=-.5ex]current bounding box.center)}]
	\draw (0,0) circle (1);
	\draw ( 30:1) -- ( 0, 0.2)
          (-30:1) -- ( 0,-0.2)
          (  0:1) -- (-60:1);
	\fill[black] ( 30:1) circle (1pt)
                 (-30:1) circle (1pt)
                 (-60:1) circle (1pt)
                 (  0:1) circle (1pt);
	\end{tikzpicture}+
	\begin{tikzpicture}[baseline={([yshift=-.5ex]current bounding box.center)}]
	\draw (0,0) circle (1);
	\draw ( 30:1) -- ( 0, 0.2)
          (-30:1) -- ( 0,-0.2)
          (15:1) .. controls (0.9,0).. (-15:1);
	\fill[black] ( 30:1) circle (1pt)
                 (-30:1) circle (1pt)
                 (-15:1) circle (1pt)
                 ( 15:1) circle (1pt);
	\end{tikzpicture}
\]
The diagram on the left is $J_{i,j}$. The first diagram on the right is $J_{i-1,j+1}$. The second and third diagrams on the right are $J_{i-1,j}$ with a leaf, respectively. The fourth diagram on the right is $J_{i-1,j}$ times an isolated chord.

Applying $\bar{w}_{\fg}$ to both sides of the equation
and using the fact that a product of two nontrivial Jacobi diagrams
projects to~$0$ in primitives, we get Theorem~\ref{th1}:
\begin{align*}
	\bar{w}_{\fg}(J_{i,j})&=\bar{w}_{\fg}(J_{i-1,j+1})+2\times\frac{1}{2\lambda}\bar{w}_{\fg}(J_{i-1,j})+0 \\
	                      &=\bar{w}_{\fg}(J_{i-1,j+1})+\frac{1}{\lambda}\bar{w}_{\fg}(J_{i-1,j}).
\end{align*}

\section{Values of the $\SL_3$ weight system on the special family $J_{i,0}$ of Jacobi diagrams}\label{s7}
Let $W(\cdot)$ denote the weight system associated to $\SL_3$ and the matrix trace as the invariant bilinear form.
Below, we will make use of the following  theorem 
about the values of this weight system on Jacobi diagrams.

\begin{theorem}[Kenichi Kuga and Shutaro Yoshizumi~\cite{yoshizumi1999some}]\label{thm7}
For Jacobi diagrams that are different only in parts depicted as below, the following relations hold:
\begin{enumerate}
\item $W(\begin{tikzpicture}[baseline={([yshift=-.5ex]current bounding box.center)}]
	\draw (0,0) circle (0.2);
	\draw ( 90:0.5) -- ( 90:0.2)
	      (-90:0.5) -- (-90:0.2);
	\fill[black]( 90:0.2) circle (1pt)
                (-90:0.2) circle (1pt);
	\end{tikzpicture})=6W(\ \begin{tikzpicture}[baseline={([yshift=-.5ex]current bounding box.center)}]
	\draw ( 90:0.5) -- (-90:0.5);
	\end{tikzpicture}\ ) $,
\item $W(\begin{tikzpicture}[baseline={([yshift=-.5ex]current bounding box.center)}]
	\draw ( 90:0.6) -- ( 90:0.3)
          (210:0.6) -- (210:0.3)
          (330:0.6) -- (330:0.3)
          (210:0.3) -- ( 90:0.3)
          ( 90:0.3) -- (330:0.3)
          (330:0.3) -- (210:0.3);
	\fill[black]( 90:0.3) circle (1pt)
                (210:0.3) circle (1pt)
                (330:0.3) circle (1pt);
	\end{tikzpicture})=3W(\begin{tikzpicture}[baseline={([yshift=-.5ex]current bounding box.center)}]
	\draw ( 90:0.6) -- (0,0)
          (210:0.6) -- (0,0)
          (330:0.6) -- (0,0);
	\fill[black](0,0) circle (1pt);
	\end{tikzpicture}) $,
\item $W(\begin{tikzpicture}[baseline={([yshift=-.5ex]current bounding box.center)}]
	\draw ( 30:0.6) -- ( 0.2, 0.2)
          (-30:0.6) -- ( 0.2,-0.2)
          (150:0.6) -- (-0.2, 0.2)
          (210:0.6) -- (-0.2,-0.2)
          (0.2,0.2) -- (0.2,-0.2)
          (0.2,0.2) -- (-0.2,0.2)
          (-0.2,-0.2) -- (0.2,-0.2)
          (-0.2,-0.2) -- (-0.2,0.2);
	\fill[black] ( 0.2, 0.2) circle (1pt)
                 (-0.2,-0.2) circle (1pt)
                 (-0.2,0.2) circle (1pt)
                 (0.2,-0.2) circle (1pt);
	\end{tikzpicture})=W(\begin{tikzpicture}[baseline={([yshift=-.5ex]current bounding box.center)}]
	\draw ( 30:0.6) -- ( 0, 0.2)
          (-30:0.6) -- ( 0,-0.2)
          (150:0.6) -- (-0, 0.2)
          (210:0.6) -- (-0,-0.2)
          (-0,0.2) -- (-0,-0.2);
	\fill[black] ( 0, 0.2) circle (1pt)
                 (-0,-0.2) circle (1pt);
	\end{tikzpicture})+W(\begin{tikzpicture}[baseline={([yshift=-.5ex]current bounding box.center)}]
	\draw ( 30:0.6) -- ( 0.3, 0)
          (-30:0.6) -- ( 0.3,-0)
          (150:0.6) -- (-0.3, 0)
          (210:0.6) -- (-0.3,-0)
          (-0.3,0) -- (0.3,-0);
	\fill[black] ( 0.3, 0) circle (1pt)
                 (-0.3,-0) circle (1pt);
	\end{tikzpicture})+3W(\begin{tikzpicture}[baseline={([yshift=-.5ex]current bounding box.center)}]
	\draw ( 30:0.6)  ..  controls ( 0, 0.2) and  (0,-0.2)  .. (-30:0.6)
          (150:0.6) ..  controls ( 0, 0.2) and  (0,-0.2)  .. (210:0.6);
	\end{tikzpicture})+\\\phantom{W(\begin{tikzpicture}[baseline={([yshift=-.5ex]current bounding box.center)}]
	\draw ( 30:0.6) -- ( 0.2, 0.2)
          (-30:0.6) -- ( 0.2,-0.2)
          (150:0.6) -- (-0.2, 0.2)
          (210:0.6) -- (-0.2,-0.2)
          (0.2,0.2) -- (0.2,-0.2)
          (0.2,0.2) -- (-0.2,0.2)
          (-0.2,-0.2) -- (0.2,-0.2)
          (-0.2,-0.2) -- (-0.2,0.2);
	\fill[black] ( 0.2, 0.2) circle (1pt)
                 (-0.2,-0.2) circle (1pt)
                 (-0.2,0.2) circle (1pt)
                 (0.2,-0.2) circle (1pt);
	\end{tikzpicture})=}+3W(\begin{tikzpicture}[baseline={([yshift=-.5ex]current bounding box.center)}]
	\draw ( 30:0.6)  ..  controls ( 0.3, 0) and  (-0.3, 0)  .. (150:0.6)
          (-30:0.6) ..  controls ( 0.3, 0) and  (-0.3, 0)  .. (210:0.6);
	\end{tikzpicture})+3W(\begin{tikzpicture}[baseline={([yshift=-.5ex]current bounding box.center)}]
	\draw ( 30:0.6)  -- (210:0.6)
          (-30:0.6)  -- (150:0.6);
	\end{tikzpicture})$.
\end{enumerate}
\end{theorem}

In particular, for the Jacobi diagrams $J_{i,0}$, the above formulas yield
\begin{align*}
W(J_{i,0})&=W(\begin{tikzpicture}[baseline={([yshift=-.5ex]current bounding box.center)}]
	\draw (0,0) circle (1);
	\draw ( 30:1) -- ( 0.6, 0.2)
          (-30:1) -- ( 0.6,-0.2)
          (150:1) -- (-0.6, 0.2)
          (210:1) -- (-0.6,-0.2)
          ( 0.6, 0.2) -- ( 0.6,-0.2)
          (-0.6, 0.2) -- (-0.6,-0.2)
          ( 0.6, 0.2) -- ( 0.1, 0.2)
          ( 0.6,-0.2) -- ( 0.1,-0.2)
          (-0.6, 0.2) -- (-0.1, 0.2)
          (-0.6,-0.2) -- (-0.1,-0.2)
          ( 0.4, 0.2) -- ( 0.4,-0.2)
          (-0.4, 0.2) -- (-0.4,-0.2)
          ( 0.2, 0.2) -- ( 0.2,-0.2)
          (-0.2, 0.2) -- (-0.2,-0.2);
	\fill[black] ( 0.6, 0.2) circle (1pt)
                 (-0.6, 0.2) circle (1pt)
                 ( 0.6,-0.2) circle (1pt)
                 (-0.6,-0.2) circle (1pt)
                 ( 0.4, 0.2) circle (1pt)
                 (-0.4, 0.2) circle (1pt)
                 ( 0.4,-0.2) circle (1pt)
                 (-0.4,-0.2) circle (1pt)
                 ( 0.2, 0.2) circle (1pt)
                 (-0.2, 0.2) circle (1pt)
                 ( 0.2,-0.2) circle (1pt)
                 (-0.2,-0.2) circle (1pt)
                 ( 30:1) circle (1pt)
                 (-30:1) circle (1pt)
                 (150:1) circle (1pt)
                 (210:1) circle (1pt);
    \node at (0,0) {\ldots};
    \node [below,scale=0.75] at ( 0.6,-0.2) {\tiny $i$};
    \node [below,scale=0.75] at ( 0.4,-0.2) {\tiny i-$1$};
    \node [below,scale=0.75] at ( 0.2,-0.2) {\tiny i-$2$};
    \node [below,scale=0.75] at (-0.2,-0.2) {\tiny $3$};
    \node [below,scale=0.75] at (-0.4,-0.2) {\tiny $2$};
    \node [below,scale=0.75] at (-0.6,-0.2) {\tiny $1$};
\end{tikzpicture})\\
&=W(\begin{tikzpicture}[baseline={([yshift=-.5ex]current bounding box.center)}]
	\draw (0,0) circle (1);
	\draw ( 30:1) -- ( 0.4, 0.2)
          (-30:1) -- ( 0.4,-0.2)
          (150:1) -- (-0.6, 0.2)
          (210:1) -- (-0.6,-0.2)
          (-0.6, 0.2) -- (-0.6,-0.2)
          ( 0.4, 0.2) -- ( 0.1, 0.2)
          ( 0.4,-0.2) -- ( 0.1,-0.2)
          (-0.6, 0.2) -- (-0.1, 0.2)
          (-0.6,-0.2) -- (-0.1,-0.2)
          ( 0.4, 0.2) -- ( 0.4,-0.2)
          (-0.4, 0.2) -- (-0.4,-0.2)
          ( 0.2, 0.2) -- ( 0.2,-0.2)
          (-0.2, 0.2) -- (-0.2,-0.2);
	\fill[black] 
    			 (-0.6, 0.2) circle (1pt)
    			 (-0.6,-0.2) circle (1pt)
                 ( 0.4, 0.2) circle (1pt)
                 (-0.4, 0.2) circle (1pt)
                 ( 0.4,-0.2) circle (1pt)
                 (-0.4,-0.2) circle (1pt)
                 ( 0.2, 0.2) circle (1pt)
                 (-0.2, 0.2) circle (1pt)
                 ( 0.2,-0.2) circle (1pt)
                 (-0.2,-0.2) circle (1pt)
                 ( 30:1) circle (1pt)
                 (-30:1) circle (1pt)
                 (150:1) circle (1pt)
                 (210:1) circle (1pt);
    \node at (0,0) {\ldots};
    \node [below,scale=0.75] at ( 0.4,-0.2) {\tiny i-$1$};
    \node [below,scale=0.75] at ( 0.2,-0.2) {\tiny i-$2$};
    \node [below,scale=0.75] at (-0.2,-0.2) {\tiny $3$};
    \node [below,scale=0.75] at (-0.4,-0.2) {\tiny $2$};
    \node [below,scale=0.75] at (-0.6,-0.2) {\tiny $1$};
\end{tikzpicture})+W(\begin{tikzpicture}[baseline={([yshift=-.5ex]current bounding box.center)}]
	\draw (0,0) circle (1);
	\draw 
          (150:1) -- (-0.6, 0.2)
          (210:1) -- (-0.6,-0.2)
          (-0.6, 0.2) -- (-0.6,-0.2)
          ( 0.2, 0.2) -- ( 0.1, 0.2)
          ( 0.2,-0.2) -- ( 0.1,-0.2)
          (-0.6, 0.2) -- (-0.1, 0.2)
          (-0.6,-0.2) -- (-0.1,-0.2)
          (-0.4, 0.2) -- (-0.4,-0.2)
          ( 0.2, 0.2) -- ( 0.2,-0.2)
          (-0.2, 0.2) -- (-0.2,-0.2)
          ( 0.4, 0  ) -- ( 0.2, 0.2)
          ( 0.4, 0  ) -- ( 0.2,-0.2)
          ( 0.4, 0  ) -- ( 0.6,   0)
          ( 30:1) -- ( 0.6,   0)
          (-30:1) -- ( 0.6,   0);
	\fill[black] 
    			 (-0.6, 0.2) circle (1pt)
    			 (-0.6,-0.2) circle (1pt)
                 ( 0.4, 0  ) circle (1pt)
                 ( 0.6,   0) circle (1pt)
                 (-0.4,-0.2) circle (1pt)
                 ( 0.2, 0.2) circle (1pt)
                 (-0.2, 0.2) circle (1pt)
                 ( 0.2,-0.2) circle (1pt)
                 (-0.2,-0.2) circle (1pt)
                 ( 30:1) circle (1pt)
                 (-30:1) circle (1pt)
                 (150:1) circle (1pt)
                 (210:1) circle (1pt);
    \node at (0,0) {\ldots};
    \node [below,scale=0.75] at ( 0.2,-0.2) {\tiny i-$2$};
    \node [below,scale=0.75] at (-0.2,-0.2) {\tiny $3$};
    \node [below,scale=0.75] at (-0.4,-0.2) {\tiny $2$};
    \node [below,scale=0.75] at (-0.6,-0.2) {\tiny $1$};
\end{tikzpicture})+3W(\begin{tikzpicture}[baseline={([yshift=-.5ex]current bounding box.center)}]
	\draw (0,0) circle (1);
	\draw ( 30:1) ..  controls ( 0.5, 0.2) and  (0.5,-0.2)  .. (-30:1)
          (150:1) -- (-0.6, 0.2)
          (210:1) -- (-0.6,-0.2)
          (-0.6, 0.2) -- (-0.6,-0.2)
          ( 0.2, 0.2) -- ( 0.1, 0.2)
          ( 0.2,-0.2) -- ( 0.1,-0.2)
          (-0.6, 0.2) -- (-0.1, 0.2)
          (-0.6,-0.2) -- (-0.1,-0.2)
          (-0.4, 0.2) -- (-0.4,-0.2)
          ( 0.2, 0.2) -- ( 0.2,-0.2)
          ( 0.2, 0.2) ..  controls ( 0.5, 0.2) and  (0.5,-0.2)  .. ( 0.2,-0.2)
          (-0.2, 0.2) -- (-0.2,-0.2);
	\fill[black] 
    			 (-0.6, 0.2) circle (1pt)
    			 (-0.6,-0.2) circle (1pt)
                 (-0.4, 0.2) circle (1pt)
                 (-0.4,-0.2) circle (1pt)
                 ( 0.2, 0.2) circle (1pt)
                 (-0.2, 0.2) circle (1pt)
                 ( 0.2,-0.2) circle (1pt)
                 (-0.2,-0.2) circle (1pt)
                 ( 30:1) circle (1pt)
                 (-30:1) circle (1pt)
                 (150:1) circle (1pt)
                 (210:1) circle (1pt);
    \node at (0,0) {\ldots};
    \node [below,scale=0.75] at ( 0.2,-0.2) {\tiny i-$2$};
    \node [below,scale=0.75] at (-0.2,-0.2) {\tiny $3$};
    \node [below,scale=0.75] at (-0.4,-0.2) {\tiny $2$};
    \node [below,scale=0.75] at (-0.6,-0.2) {\tiny $1$};
\end{tikzpicture})\\
&\phantom{=}+3W(\begin{tikzpicture}[baseline={([yshift=-.5ex]current bounding box.center)}]
	\draw (0,0) circle (1);
	\draw ( 30:1) -- ( 0.2, 0.2)
          (-30:1) -- ( 0.2,-0.2)
          (150:1) -- (-0.6, 0.2)
          (210:1) -- (-0.6,-0.2)
          (-0.6, 0.2) -- (-0.6,-0.2)
          ( 0.2, 0.2) -- ( 0.1, 0.2)
          ( 0.2,-0.2) -- ( 0.1,-0.2)
          (-0.6, 0.2) -- (-0.1, 0.2)
          (-0.6,-0.2) -- (-0.1,-0.2)
          (-0.4, 0.2) -- (-0.4,-0.2)
          ( 0.2, 0.2) -- ( 0.2,-0.2)
          (-0.2, 0.2) -- (-0.2,-0.2);
	\fill[black] 
    			 (-0.6, 0.2) circle (1pt)
    			 (-0.6,-0.2) circle (1pt)
                 (-0.4, 0.2) circle (1pt)
                 (-0.4,-0.2) circle (1pt)
                 ( 0.2, 0.2) circle (1pt)
                 (-0.2, 0.2) circle (1pt)
                 ( 0.2,-0.2) circle (1pt)
                 (-0.2,-0.2) circle (1pt)
                 ( 30:1) circle (1pt)
                 (-30:1) circle (1pt)
                 (150:1) circle (1pt)
                 (210:1) circle (1pt);
    \node at (0,0) {\ldots};
    \node [below,scale=0.75] at ( 0.2,-0.2) {\tiny i-$2$};
    \node [below,scale=0.75] at (-0.2,-0.2) {\tiny $3$};
    \node [below,scale=0.75] at (-0.4,-0.2) {\tiny $2$};
    \node [below,scale=0.75] at (-0.6,-0.2) {\tiny $1$};
\end{tikzpicture})+3W(\begin{tikzpicture}[baseline={([yshift=-.5ex]current bounding box.center)}]
	\draw (0,0) circle (1);
	\draw ( 30:1) -- ( 0.2,-0.2)
          (-30:1) -- ( 0.2, 0.2)
          (150:1) -- (-0.6, 0.2)
          (210:1) -- (-0.6,-0.2)
          (-0.6, 0.2) -- (-0.6,-0.2)
          ( 0.2, 0.2) -- ( 0.1, 0.2)
          ( 0.2,-0.2) -- ( 0.1,-0.2)
          (-0.6, 0.2) -- (-0.1, 0.2)
          (-0.6,-0.2) -- (-0.1,-0.2)
          (-0.4, 0.2) -- (-0.4,-0.2)
          ( 0.2, 0.2) -- ( 0.2,-0.2)
          (-0.2, 0.2) -- (-0.2,-0.2);
	\fill[black] 
    			 (-0.6, 0.2) circle (1pt)
    			 (-0.6,-0.2) circle (1pt)
                 (-0.4, 0.2) circle (1pt)
                 (-0.4,-0.2) circle (1pt)
                 ( 0.2, 0.2) circle (1pt)
                 (-0.2, 0.2) circle (1pt)
                 ( 0.2,-0.2) circle (1pt)
                 (-0.2,-0.2) circle (1pt)
                 ( 30:1) circle (1pt)
                 (-30:1) circle (1pt)
                 (150:1) circle (1pt)
                 (210:1) circle (1pt);
    \node at (0,0) {\ldots};
    \node [below,scale=0.75] at ( 0.2,-0.2) {\tiny i-$2$};
    \node [below,scale=0.75] at (-0.2,-0.2) {\tiny $3$};
    \node [below,scale=0.75] at (-0.4,-0.2) {\tiny $2$};
    \node [below,scale=0.75] at (-0.6,-0.2) {\tiny $1$};
\end{tikzpicture}).
\end{align*}

We know
\[
	W(\begin{tikzpicture}[baseline={([yshift=-.5ex]current bounding box.center)}]
	\draw (0,0) circle (1);
	\draw ( 30:1) -- ( 0.4, 0.2)
          (-30:1) -- ( 0.4,-0.2)
          (150:1) -- (-0.6, 0.2)
          (210:1) -- (-0.6,-0.2)
          (-0.6, 0.2) -- (-0.6,-0.2)
          ( 0.4, 0.2) -- ( 0.1, 0.2)
          ( 0.4,-0.2) -- ( 0.1,-0.2)
          (-0.6, 0.2) -- (-0.1, 0.2)
          (-0.6,-0.2) -- (-0.1,-0.2)
          ( 0.4, 0.2) -- ( 0.4,-0.2)
          (-0.4, 0.2) -- (-0.4,-0.2)
          ( 0.2, 0.2) -- ( 0.2,-0.2)
          (-0.2, 0.2) -- (-0.2,-0.2);
	\fill[black] 
    			 (-0.6, 0.2) circle (1pt)
    			 (-0.6,-0.2) circle (1pt)
                 ( 0.4, 0.2) circle (1pt)
                 (-0.4, 0.2) circle (1pt)
                 ( 0.4,-0.2) circle (1pt)
                 (-0.4,-0.2) circle (1pt)
                 ( 0.2, 0.2) circle (1pt)
                 (-0.2, 0.2) circle (1pt)
                 ( 0.2,-0.2) circle (1pt)
                 (-0.2,-0.2) circle (1pt)
                 ( 30:1) circle (1pt)
                 (-30:1) circle (1pt)
                 (150:1) circle (1pt)
                 (210:1) circle (1pt);
    \node at (0,0) {\ldots};
    \node [below,scale=0.75] at ( 0.4,-0.2) {\tiny i-$1$};
    \node [below,scale=0.75] at ( 0.2,-0.2) {\tiny i-$2$};
    \node [below,scale=0.75] at (-0.2,-0.2) {\tiny $3$};
    \node [below,scale=0.75] at (-0.4,-0.2) {\tiny $2$};
    \node [below,scale=0.75] at (-0.6,-0.2) {\tiny $1$};
\end{tikzpicture})=W(J_{i-1,0}) ,\ \  W(\begin{tikzpicture}[baseline={([yshift=-.5ex]current bounding box.center)}]
	\draw (0,0) circle (1);
	\draw ( 30:1) -- ( 0.2, 0.2)
          (-30:1) -- ( 0.2,-0.2)
          (150:1) -- (-0.6, 0.2)
          (210:1) -- (-0.6,-0.2)
          (-0.6, 0.2) -- (-0.6,-0.2)
          ( 0.2, 0.2) -- ( 0.1, 0.2)
          ( 0.2,-0.2) -- ( 0.1,-0.2)
          (-0.6, 0.2) -- (-0.1, 0.2)
          (-0.6,-0.2) -- (-0.1,-0.2)
          (-0.4, 0.2) -- (-0.4,-0.2)
          ( 0.2, 0.2) -- ( 0.2,-0.2)
          (-0.2, 0.2) -- (-0.2,-0.2);
	\fill[black] 
    			 (-0.6, 0.2) circle (1pt)
    			 (-0.6,-0.2) circle (1pt)
                 (-0.4, 0.2) circle (1pt)
                 (-0.4,-0.2) circle (1pt)
                 ( 0.2, 0.2) circle (1pt)
                 (-0.2, 0.2) circle (1pt)
                 ( 0.2,-0.2) circle (1pt)
                 (-0.2,-0.2) circle (1pt)
                 ( 30:1) circle (1pt)
                 (-30:1) circle (1pt)
                 (150:1) circle (1pt)
                 (210:1) circle (1pt);
    \node at (0,0) {\ldots};
    \node [below,scale=0.75] at ( 0.2,-0.2) {\tiny i-$2$};
    \node [below,scale=0.75] at (-0.2,-0.2) {\tiny $3$};
    \node [below,scale=0.75] at (-0.4,-0.2) {\tiny $2$};
    \node [below,scale=0.75] at (-0.6,-0.2) {\tiny $1$};
\end{tikzpicture})=W(J_{i-2,0}) .
\]
By Theorem \ref{thm7}(2), we have 
\[
	W(\begin{tikzpicture}[baseline={([yshift=-.5ex]current bounding box.center)}]
	\draw (0,0) circle (1);
	\draw 
          (150:1) -- (-0.6, 0.2)
          (210:1) -- (-0.6,-0.2)
          (-0.6, 0.2) -- (-0.6,-0.2)
          ( 0.2, 0.2) -- ( 0.1, 0.2)
          ( 0.2,-0.2) -- ( 0.1,-0.2)
          (-0.6, 0.2) -- (-0.1, 0.2)
          (-0.6,-0.2) -- (-0.1,-0.2)
          (-0.4, 0.2) -- (-0.4,-0.2)
          ( 0.2, 0.2) -- ( 0.2,-0.2)
          (-0.2, 0.2) -- (-0.2,-0.2)
          ( 0.4, 0  ) -- ( 0.2, 0.2)
          ( 0.4, 0  ) -- ( 0.2,-0.2)
          ( 0.4, 0  ) -- ( 0.6,   0)
          ( 30:1) -- ( 0.6,   0)
          (-30:1) -- ( 0.6,   0);
	\fill[black] 
    			 (-0.6, 0.2) circle (1pt)
    			 (-0.6,-0.2) circle (1pt)
                 ( 0.4, 0  ) circle (1pt)
                 ( 0.6,   0) circle (1pt)
                 (-0.4,-0.2) circle (1pt)
                 ( 0.2, 0.2) circle (1pt)
                 (-0.2, 0.2) circle (1pt)
                 ( 0.2,-0.2) circle (1pt)
                 (-0.2,-0.2) circle (1pt)
                 ( 30:1) circle (1pt)
                 (-30:1) circle (1pt)
                 (150:1) circle (1pt)
                 (210:1) circle (1pt);
    \node at (0,0) {\ldots};
    \node [below,scale=0.75] at ( 0.2,-0.2) {\tiny i-$2$};
    \node [below,scale=0.75] at (-0.2,-0.2) {\tiny $3$};
    \node [below,scale=0.75] at (-0.4,-0.2) {\tiny $2$};
    \node [below,scale=0.75] at (-0.6,-0.2) {\tiny $1$};
\end{tikzpicture})=3^{i-2}\times W(\begin{tikzpicture}[baseline={([yshift=-.5ex]current bounding box.center)}]
	\draw (0,0) circle (1);
	\draw ( 30:1) -- ( 0.5, 0)
          (-30:1) -- ( 0.5,-0)
          (150:1) -- (-0.5, 0)
          (210:1) -- (-0.5,-0)
          ( 0.5,0) -- (-0.5,-0);
	\fill[black] ( 0.5, 0) circle (1pt)
                 (-0.5,-0) circle (1pt)
                 ( 30:1) circle (1pt)
                 (-30:1) circle (1pt)
                 (150:1) circle (1pt)
                 (210:1) circle (1pt);
	\end{tikzpicture})=3^{i}c_2.
\]
By Theorem \ref{thm7}(1), we have 
\[
	W(\begin{tikzpicture}[baseline={([yshift=-.5ex]current bounding box.center)}]
	\draw (0,0) circle (1);
	\draw ( 30:1) ..  controls ( 0.5, 0.2) and  (0.5,-0.2)  .. (-30:1)
          (150:1) -- (-0.6, 0.2)
          (210:1) -- (-0.6,-0.2)
          (-0.6, 0.2) -- (-0.6,-0.2)
          ( 0.2, 0.2) -- ( 0.1, 0.2)
          ( 0.2,-0.2) -- ( 0.1,-0.2)
          (-0.6, 0.2) -- (-0.1, 0.2)
          (-0.6,-0.2) -- (-0.1,-0.2)
          (-0.4, 0.2) -- (-0.4,-0.2)
          ( 0.2, 0.2) -- ( 0.2,-0.2)
          ( 0.2, 0.2) ..  controls ( 0.5, 0.2) and  (0.5,-0.2)  .. ( 0.2,-0.2)
          (-0.2, 0.2) -- (-0.2,-0.2);
	\fill[black] 
    			 (-0.6, 0.2) circle (1pt)
    			 (-0.6,-0.2) circle (1pt)
                 (-0.4, 0.2) circle (1pt)
                 (-0.4,-0.2) circle (1pt)
                 ( 0.2, 0.2) circle (1pt)
                 (-0.2, 0.2) circle (1pt)
                 ( 0.2,-0.2) circle (1pt)
                 (-0.2,-0.2) circle (1pt)
                 ( 30:1) circle (1pt)
                 (-30:1) circle (1pt)
                 (150:1) circle (1pt)
                 (210:1) circle (1pt);
    \node at (0,0) {\ldots};
    \node [below,scale=0.75] at ( 0.2,-0.2) {\tiny i-$2$};
    \node [below,scale=0.75] at (-0.2,-0.2) {\tiny $3$};
    \node [below,scale=0.75] at (-0.4,-0.2) {\tiny $2$};
    \node [below,scale=0.75] at (-0.6,-0.2) {\tiny $1$};
\end{tikzpicture})=6^{i-2}\times W(\begin{tikzpicture}[baseline={([yshift=-.5ex]current bounding box.center)}]
	\draw (0,0) circle (1);
	\draw ( 30:1) ..  controls ( 0.5, 0.2) and  (0.5,-0.2)  .. (-30:1)
          (150:1) ..  controls (-0.5, 0.2) and  (-0.5,-0.2)  .. (210:1);
	\fill[black] 
                 ( 30:1) circle (1pt)
                 (-30:1) circle (1pt)
                 (150:1) circle (1pt)
                 (210:1) circle (1pt);
\end{tikzpicture})=6^{i-2}c_2^2.
\]
By Theorem \ref{thm7}(2) and the STU-relations, we have
\begin{align*}
	W(\begin{tikzpicture}[baseline={([yshift=-.5ex]current bounding box.center)}]
	\draw (0,0) circle (1);
	\draw ( 30:1) -- ( 0.2,-0.2)
          (-30:1) -- ( 0.2, 0.2)
          (150:1) -- (-0.6, 0.2)
          (210:1) -- (-0.6,-0.2)
          (-0.6, 0.2) -- (-0.6,-0.2)
          ( 0.2, 0.2) -- ( 0.1, 0.2)
          ( 0.2,-0.2) -- ( 0.1,-0.2)
          (-0.6, 0.2) -- (-0.1, 0.2)
          (-0.6,-0.2) -- (-0.1,-0.2)
          (-0.4, 0.2) -- (-0.4,-0.2)
          ( 0.2, 0.2) -- ( 0.2,-0.2)
          (-0.2, 0.2) -- (-0.2,-0.2);
	\fill[black] 
    			 (-0.6, 0.2) circle (1pt)
    			 (-0.6,-0.2) circle (1pt)
                 (-0.4, 0.2) circle (1pt)
                 (-0.4,-0.2) circle (1pt)
                 ( 0.2, 0.2) circle (1pt)
                 (-0.2, 0.2) circle (1pt)
                 ( 0.2,-0.2) circle (1pt)
                 (-0.2,-0.2) circle (1pt)
                 ( 30:1) circle (1pt)
                 (-30:1) circle (1pt)
                 (150:1) circle (1pt)
                 (210:1) circle (1pt);
    \node at (0,0) {\ldots};
    \node [below,scale=0.75] at ( 0.2,-0.2) {\tiny i-$2$};
    \node [below,scale=0.75] at (-0.2,-0.2) {\tiny $3$};
    \node [below,scale=0.75] at (-0.4,-0.2) {\tiny $2$};
    \node [below,scale=0.75] at (-0.6,-0.2) {\tiny $1$};
\end{tikzpicture})&=W(\begin{tikzpicture}[baseline={([yshift=-.5ex]current bounding box.center)}]
	\draw (0,0) circle (1);
	\draw ( 30:1) -- ( 0.2, 0.2)
          (-30:1) -- ( 0.2,-0.2)
          (150:1) -- (-0.6, 0.2)
          (210:1) -- (-0.6,-0.2)
          (-0.6, 0.2) -- (-0.6,-0.2)
          ( 0.2, 0.2) -- ( 0.1, 0.2)
          ( 0.2,-0.2) -- ( 0.1,-0.2)
          (-0.6, 0.2) -- (-0.1, 0.2)
          (-0.6,-0.2) -- (-0.1,-0.2)
          (-0.4, 0.2) -- (-0.4,-0.2)
          ( 0.2, 0.2) -- ( 0.2,-0.2)
          (-0.2, 0.2) -- (-0.2,-0.2);
	\fill[black] 
    			 (-0.6, 0.2) circle (1pt)
    			 (-0.6,-0.2) circle (1pt)
                 (-0.4, 0.2) circle (1pt)
                 (-0.4,-0.2) circle (1pt)
                 ( 0.2, 0.2) circle (1pt)
                 (-0.2, 0.2) circle (1pt)
                 ( 0.2,-0.2) circle (1pt)
                 (-0.2,-0.2) circle (1pt)
                 ( 30:1) circle (1pt)
                 (-30:1) circle (1pt)
                 (150:1) circle (1pt)
                 (210:1) circle (1pt);
    \node at (0,0) {\ldots};
    \node [below,scale=0.75] at ( 0.2,-0.2) {\tiny i-$2$};
    \node [below,scale=0.75] at (-0.2,-0.2) {\tiny $3$};
    \node [below,scale=0.75] at (-0.4,-0.2) {\tiny $2$};
    \node [below,scale=0.75] at (-0.6,-0.2) {\tiny $1$};
\end{tikzpicture})-W(\begin{tikzpicture}[baseline={([yshift=-.5ex]current bounding box.center)}]
	\draw (0,0) circle (1);
	\draw 
          (150:1) -- (-0.6, 0.2)
          (210:1) -- (-0.6,-0.2)
          (-0.6, 0.2) -- (-0.6,-0.2)
          ( 0.2, 0.2) -- ( 0.1, 0.2)
          ( 0.2,-0.2) -- ( 0.1,-0.2)
          (-0.6, 0.2) -- (-0.1, 0.2)
          (-0.6,-0.2) -- (-0.1,-0.2)
          (-0.4, 0.2) -- (-0.4,-0.2)
          ( 0.2, 0.2) -- ( 0.2,-0.2)
          (-0.2, 0.2) -- (-0.2,-0.2)
          ( 0.4, 0  ) -- ( 0.2, 0.2)
          ( 0.4, 0  ) -- ( 0.2,-0.2)
          ( 0.4, 0  ) -- ( 1,   0);
	\fill[black] 
    			 (-0.6, 0.2) circle (1pt)
    			 (-0.6,-0.2) circle (1pt)
                 ( 0.4, 0  ) circle (1pt)
                 ( 1,   0) circle (1pt)
                 (-0.4,-0.2) circle (1pt)
                 ( 0.2, 0.2) circle (1pt)
                 (-0.2, 0.2) circle (1pt)
                 ( 0.2,-0.2) circle (1pt)
                 (-0.2,-0.2) circle (1pt)
                 (150:1) circle (1pt)
                 (210:1) circle (1pt);
    \node at (0,0) {\ldots};
    \node [below,scale=0.75] at ( 0.2,-0.2) {\tiny i-$2$};
    \node [below,scale=0.75] at (-0.2,-0.2) {\tiny $3$};
    \node [below,scale=0.75] at (-0.4,-0.2) {\tiny $2$};
    \node [below,scale=0.75] at (-0.6,-0.2) {\tiny $1$};
\end{tikzpicture}) \\
&=W(J_{i-2,0})-3^{i-2}W(\begin{tikzpicture}[baseline={([yshift=-.5ex]current bounding box.center)}]
	\draw (0,0) circle (1);
	\draw (120:1)--(0,0)
	      (0:1)--(0,0)
	      (240:1)--(0,0);
	\fill[black] (120:1) circle (1pt)
				 (1,0) circle (1pt)
				 (0,0) circle (1pt)
                 (240:1) circle (1pt);
\end{tikzpicture})) \\
&=W(J_{i-2,0})-3^{i-1}c_2.
\end{align*}

Combining all these together, we get 
\[
	W(J_{i,0})=W(J_{i-1,0})+6W(J_{i-2,0})+3\times6^{i-2}c_2^2.
\]

Now, we have $W(J_{1,0})=9c_2$ and $W(J_{2,0})=9c_2(c_2+9)$. By induction, we get
\[
	W(J_{i,0})=c_2\left(\left(\frac{6^{i}}{8}+\frac{27}{40}(-2)^{i}+\frac{1}{5}3^{i}\right)c_2+\frac{9}{5}\left(3^{i}-(-2)^{i}\right)\right)\text{\ for \ } (i\ge 1) 
\]
and $W(J_{0,0}):=0$, we have the  exponential generating functions form
\[
	\sum_{n=0} W(J_{n,0})\frac{x^n}{n!}=c_2\left((\frac{1}{8}e^{6x}+\frac{27}{40}e^{-2x}+\frac{1}{5}e^{3x})c_2+\frac{9}{5}(e^{3x}-e^{-2x})\right)-c_2^2.
\]
By Corollary~\ref{cor1}(3), we conclude that
\begin{align*}
	\sum_{n=0} \bar{w}_{\SL_3}(K_{2,n})\frac{x^n}{n!}&=\sum_{n=0} \bar{w}_{\SL_3}(J_{0,n})\frac{x^n}{n!}=e^{-6x}\sum_{n=0} \bar{w}_{\SL_3}(J_{n,0})\frac{x^n}{n!}; \\
	\sum_{n=0} \bar{w}_{\SL_3}(K_{2,n})\frac{x^n}{n!}&=\frac{c_2}{40}\left((27c_2-72)e^{-8x}+(8c_2+72)e^{-3x}-40c_2e^{-6x}+5c_2 \right).
\end{align*}

Now we can reconstruct the values of the $\SL_3$ weight system
on the chord diagrams $J_{2,n}$ by making use of a result in~\cite{Filippova2020}, which says
$$\sum_{n=0} \pi(K_{2,n})\frac{x^n}{n!}=e^{-K_{0,1}x}\sum_{n=0} K_{2,n}\frac{x^n}{n!}-\left(\sum_{n=0} \pi(K_{1,n})\frac{x^n}{n!}\right)^2.
$$
We get the values of $\SL_3$ weight system on chord diagrams whose intersection graph is complete bipartite $K_{2,n}$.
\begin{theorem} We have
\begin{align*}
\sum_{n=0} w_{\SL_3}(K_{2,n})\frac{x^n}{n!}&=\frac{c_2}{40}\left((27c_2-72)e^{(c_2-8)x}+(8c_2+72)e^{(c_2-3)x}+5c_2e^{c_2x}\right),\\ 
w_{\SL_3}(K_{2,n})&=\frac{c_2}{40}\left((27c_2-72)(c_2-8)^n+(8c_2+72)(c_2-3)^n+5c_2^{n+1} \right).
\end{align*}
\end{theorem}

Our computations show,
in particular, that, for the chord diagrams
whose intersection graph is~$K_{2,n}$, their
projection to the subspace of primitives
can be represented as a linear combination
of connected Jacobi diagrams with at most~$4$ legs.
It has been conjectured earlier by S.~Lando
that the value of the weight system~$\SL_2$ on a projection to the subspace of primitives 
of a chord diagram is a polynomial in the quadratic
Casimir element~$c$ whose degree does not exceed
half the length of the largest cycle in the
intersection graph of the chord diagram.
There is a lot of evidence supporting
this conjecture, see, for example~\cite{Filippova2020,Filippova2021}.
The following more general conjecture may
explain Lando's one, and is, probably, easier to prove.

\begin{conjecture}[S.~Lando, Z.~Yang]
	Let~$D$ be a chord diagram, and let~$\ell$
	be the length of the longest cycle in it.
	Then the projection $\pi(D)$ of~$D$
	to the subspace of primitives is a linear
	combination of connected Jacobi diagrams
	with at most $\ell$ legs.
	In particular, the value of a weight system
	associated to an arbitrary Lie algebra~$\fg$
	on this projection has degree at most $\ell$.
\end{conjecture}


\begin{thebibliography}{10}

\bibitem{bar1995vassiliev}
Dror Bar-Natan.
\newblock On the Vassiliev knot invariants.
\newblock {\em Topology}, 34(2):423--472, 1995.
\newblock (an updated version available at
  \url{http://www.math.toronto.edu/~drorbn/papers}).

\bibitem{Chmutov2012}
S.~Chmutov, S.~Duzhin, and J.~Mostovoy.
\newblock {\em Introduction to {Vassiliev Knot Invariants}}.
\newblock Cambridge University Press, May 2012.

\bibitem{chmutov1997remarks}
S.~Chmutov and A.~Varchenko.
\newblock Remarks on the Vassiliev knot invariants coming from
  $\mathfrak{sl}_2$.
\newblock {\em Topology}, 36(1):153--178, 1997.

\bibitem{Chmutov2007}
S.~V. Chmutov and S.~K. Lando.
\newblock Mutant knots and intersection graphs.
\newblock {\em Algebr. Geom. Topol.}, pages 1579--1598, July 2007.

\bibitem{Filippova2020}
P.~A. Filippova.
\newblock Values of the $\SL_2$ weight system on complete bipartite graphs.
\newblock {\em Funct. Anal. Appl.}, pages 73--93, 54(3) 2020.

\bibitem{Filippova2021}
P.~A. Filippova.
\newblock Values of the $\SL_2$ weight system on a family of graphs
that are nor intersection graphs of chord diagrams.
\newblock submitted 2021.


\bibitem{kon1993}
M.~Kontsevich.
\newblock Vassiliev knot invariants.
\newblock in: {\em Advances in Soviet Math.}, 16(2):137--150, 1993.

\bibitem{lando2000hopf}
Sergei~K. Lando.
\newblock On a Hopf algebra in graph theory.
\newblock {\em Journal of Combinatorial Theory, Series B}, 80(1):104--121,
  2000.

\bibitem{lando2013graphs}
Sergei~K. Lando and Alexander~K. Zvonkin.
\newblock {\em Graphs on surfaces and their applications}, volume 141.
\newblock Springer Science \& Business Media, 2013.

\bibitem{schmitt1994incidence}
William~R Schmitt.
\newblock Incidence Hopf algebras.
\newblock {\em Journal of Pure and Applied Algebra}, 96(3):299--330, 1994.

\bibitem{vassiliev1990cohomology}
V.~A. Vassiliev.
\newblock {\em Cohomology of knot spaces}.
\newblock in: Advances in Soviet Math., 1990.

\bibitem{yoshizumi1999some}
Shutaro Yoshizumi, Kenichi Kuga, et~al.
\newblock Some formulas in $\SL_3 $ weight systems.
\newblock {\em Proceedings of the Japan Academy, Series A, Mathematical
  Sciences}, 75(7):134--136, 1999.

\end{thebibliography}
\end{document}